\documentclass[format=acmsmall, authorversion, screen=true]{acmart}
\usepackage{a26}
\usepackage{booktabs} 
\usepackage[ruled]{algorithm2e} 

\usepackage[utf8]{inputenc} 
\usepackage[T1]{fontenc}    
\usepackage{hyperref}       
\usepackage{url}            
\usepackage{booktabs}       
\usepackage{amsfonts}       
\usepackage{nicefrac}       
\usepackage{microtype}      
\usepackage{lipsum}
\usepackage{fancyhdr}       
\usepackage{graphicx}       
\graphicspath{{media/}}     
\usepackage{amsmath}

\usepackage{cuted} 

\usepackage{multirow}
\usepackage{array}
\usepackage{longtable}

\usepackage{gensymb}

\usepackage{caption}
\usepackage{subcaption}

\newtheorem{assumption}{Assumption}

\SetAlFnt{\small}
\SetAlCapFnt{\small}
\SetAlCapNameFnt{\small}
\SetAlCapHSkip{0pt}
\IncMargin{-\parindent}
\newcommand{\muf}{\mu_f}

\def\R{\ensuremath{\mathbb{R}}}
\def\N{\ensuremath{\mathbb{N}}}

\newcommand{\calm}{\mathcal{M}}
\renewcommand{\phi}{\varphi}
\newcommand{\eps}{\varepsilon}
\def\dmu{\mathrm{d}\mu}
\def\dnu{\mathrm{d}\nu}
\def\dl{\mathrm{d}\lambda}
\def\dd{\mathrm{d}\delta}

\makeatletter
\let\orig@typeset@author@bx\typeset@author@bx
\def\typeset@author@bx{\bgroup\let\if@ACM@anonymous\iffalse\orig@typeset@author@bx\egroup}
\makeatother

\title[Natural Invariant Measures]{Natural Invariant Measures for Chaotic Game Dynamics: Finding Order in Chaos}

\author{Jakub Bielawski}
\orcid{https://orcid.org/0000-0001-6816-4501}
\affiliation{
\institution{Department of Mathematics, Krakow University of Economics, }
 Rakowicka 27, 31-510 Kraków  
  \city{Kraków}
  \country{Poland}, bielawsj@uek.krakow.pl}

\author{Thiparat Chotibut}
\orcid{https://orcid.org/0000-0002-8936-1413}
\affiliation{\institution{Chula Intelligent and Complex Systems Lab, Department of Physics, Faculty of Science, Chulalongkorn University,} Bangkok 10330 \city{ Bangkok} \country{Thailand}, thiparatc@gmail.com}
\email{thiparat.c@chula.ac.th}%

\author{Fryderyk Falniowski}
\orcid{https://orcid.org/0000-0002-1994-9019}
\affiliation{\institution{Department of Mathematics, Krakow University of Economics, }
Rakowicka 27, 31-510 Kraków  
  \city{Kraków}
  \country{Poland}, falniowf@uek.krakow.pl}
\email{falniowf@uek.krakow.pl}%

\author{Micha{\l} Misiurewicz}
\orcid{https://orcid.org/0000-0003-2932-8686}
\affiliation{\institution{Department of Mathematical Sciences, Indiana University Indianapolis,} 402 N. Blackford Street, Indianapolis, Indiana 46202 \city{ Indiana} \country{USA}, mmisiure@iu.edu}
\email{mmisiure@iu.edu}%

\author{Georgios Piliouras}
\orcid{https://orcid.org/0000-0002-6236-3566}
\affiliation{\institution{Google DeepMind, } London EC4A 3TW \city{ London} \country{United Kingdom}, gpil@google.com}
\email{gpil@google.com}%

\begin{document}

\begin{abstract}
We study the long-term behavior of the Multiplicative Weights Update (MWU) algorithm in game settings where learning dynamics frequently fail to converge to Nash equilibria and instead exhibit Li-Yorke chaos. While such chaos precludes the prediction of specific long-term strategy profiles, it does not imply a lack of statistical structure. We demonstrate that \emph{natural invariant measures}---a fundamental concept from ergodic theory---provide the rigorous framework necessary to find order within this chaos. Focusing on a two-strategy congestion game, we prove that these measures allow for a comprehensive statistical characterization of the dynamics. Crucially, we show that this framework extends beyond simple strategy frequencies to \emph{general observables}, enabling the precise calculation of long-term time averages for broad classes of economic metrics---including \emph{payoffs}, \emph{social cost}, and \emph{regret}---despite chaos. Our results reveal that this simple learning algorithm captures the full spectrum of behaviors found in one-dimensional dynamical systems, from unique or multiple absolutely continuous invariant measures to complex periodic attractors as well as coexisting chaotic and stable (periodic) behaviors. By bridging game theory and dynamical systems, we show that statistical predictability is attainable even in the absence of pointwise convergence.
\end{abstract}

\begin{titlepage}

\begin{teaserfigure}
  \centering
\centerline{\includegraphics[width=0.65\textwidth]{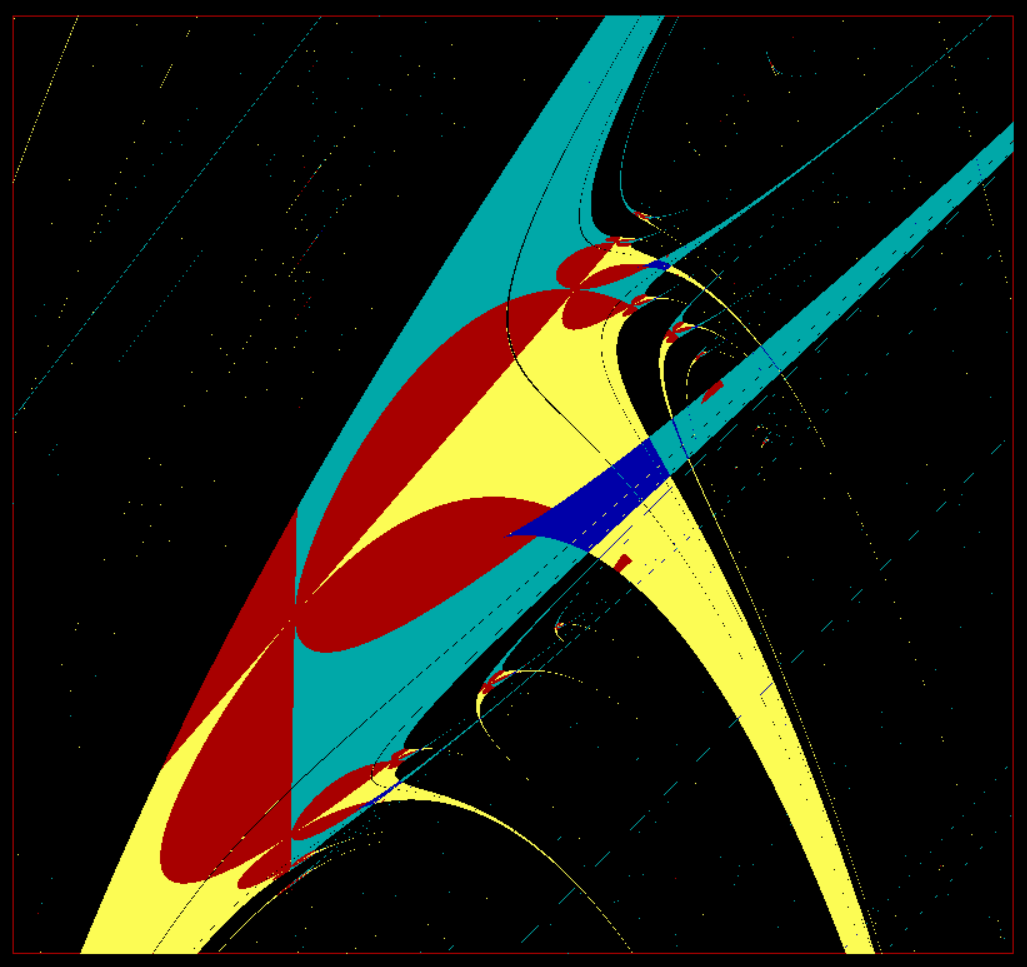}}
 \vskip 0.2in
  \Description{}
  \label{fig:teaser}
\end{teaserfigure}

\maketitle
\makeatletter
\gdef\@ACM@checkaffil{}
\makeatother

\end{titlepage}


\section{Introduction}

A central goal of game theory is to predict the long-term behavior of systems where multiple strategic agents interact. While traditional equilibrium concepts, such as Nash equilibrium, are standard tools for this analysis, they often fall short. They can mischaracterize true dynamical stability even when strategies globally converge, and they struggle particularly when the dynamics become chaotic~\cite{piliouras2026paradoxes}. In such scenarios, strategies may exhibit persistent, unpredictable fluctuations, making it impossible to forecast the system's behavior. This is particularly relevant when agents employ learning algorithms to adjust their strategies over time, as even simple algorithms can give rise to complex, chaotic dynamics in numerous diverse game settings~\cite{p2023no,palaiopanos2017multiplicative,Sato02042002,CFMP2019,andrade2021learning,skyrms1992chaos,mukhopadhyay2020deciphering,GallaFarmer_ScientificReport18,falniowski2025discrete,mitchener2004chaos,nowak1992evolutionary,peixe2022persistent,anagnostides2026chaos,metaxas2026risk}. A paradigmatic example of such a dynamic is the Multiplicative Weights Update (MWU) algorithm, which is widely used in online learning, optimization, and game theory~\cite{Arora05themultiplicative,Nisan:2007:AGT:1296179}.

Understanding and characterizing chaotic behavior in games is of paramount importance for several reasons. First, it allows us to make more realistic predictions in a variety of applications where strategic interactions are coupled with learning and adaptation, such as in online markets~\cite{leonardos2021dynamical}, algorithmic trading, algorithmic pricing \cite{calvano2020artificial,klein2021autonomous}, and evolutionary dynamics~\cite{Sato02042002,Sandholm10}. Second, it provides insight into the design of more robust and stable mechanisms and algorithms for various game-theoretic settings. By anticipating the potential for chaotic behavior of game dynamics, we can develop strategies to mitigate their negative effects or even exploit them for beneficial purposes~\cite{nagarajan2020chaos,leonardos2023optimality,piliouras2023multi}. Third, studying chaos in games sheds light on the fundamental properties of strategic interactions and the limitations of traditional rationality assumptions~\cite{arthur1994inductive,farmer2009economy,levy1994chaos}.

Despite the ubiquity of the MWU algorithm and its tendency to generate complex dynamics in various games~\cite{chotibut2021family,lin2024no,palaiopanos2017multiplicative,bailey2018multiplicative,cheung2020chaos,cheung2022evolution,piliouras2023multi}, a comprehensive understanding of its long-term behavior in these settings remains elusive. Prior work has established that MWU dynamics can exhibit Li-Yorke chaos in certain classes of games, particularly congestion games~\cite{palaiopanos2017multiplicative,chotibut2021family,CFMP2019}. This means that the system might be highly sensitive to initial conditions and arbitrarily close initial strategy profiles can diverge exponentially over time. However, while the existence of Li-Yorke chaos demonstrates the \textit{possibility} of complex behavior, it does not provide a complete statistical characterization of the dynamics. In particular, it does not tell us how frequently different strategy profiles will be visited in the long run or how to compute time averages of important quantities like payoffs or social welfare. Despite significant follow-up work on the MWU dynamics and its generalizations in various games \cite{bielawski2021follow,bielawski2024memory,bielawski2025heterogeneity}, these fundamental questions about its statistical behavior have remained unanswered until now.

We demonstrate that natural invariant measures provide a powerful and versatile framework for studying chaotic game dynamics, specifically in the context of MWU. Our theoretical results establish a rigorous connection between these measures and the long-term statistical properties of strategy distributions. Crucially, we show that this framework extends beyond simple strategy frequencies to general observables. By analyzing the natural invariant measures associated with the MWU dynamics, we can precisely calculate the long-term time averages of a broad class of economic metrics including payoffs, social cost, and regret even when the system exhibits Li-Yorke chaos. We provide a detailed analysis of a two-strategy congestion game, a canonical setting where MWU has been shown to exhibit chaotic behavior. Our results reveal a rich variety of dynamical phenomena, including cases with unique or multiple absolutely continuous invariant measures, corresponding to different long-term statistical behaviors. This work provides, for the first time, a comprehensive statistical description of the MWU dynamics in this challenging setting, opening up new avenues for research at the intersection of game theory, learning theory, dynamical systems and statistics.

The remainder of this paper is organized as follows.
In Section \ref{sec:prelim} we provide the relevant background on dynamical systems and ergodic theory.
In Section \ref{sec:natural}, we formally define invariant measures and natural invariant measures and discuss their properties.
In Section \ref{sec:theo} we show the relations between convergence of the averages of the observables and the behavior of the invariant measure.
In Section \ref{sec:MWU}, we focus on the MWU algorithm and the associated game dynamics in two strategy nonatomic congestion games. In Section \ref{sec:theoMWU} we show how the general statements of Section \ref{sec:theo} can be applied for MWU dynamics. Then, after characterization of the range of possible behaviors, we focus on describing natural measures for MWU dynamics and present numerical experiments exemplifying behaviors in Sections \ref{sec:oneinvariant} and \ref{sec:2cycles}. 
Looking at social cost and regret in Section \ref{sec:2cycles}, we give examples of economic observables in our game dynamics.
Finally, we conclude the paper and discuss some open questions and future research directions in Section \ref{sec:discussion}.

\section{Background from dynamical systems and ergodic theory}
\label{sec:prelim}
We begin with a short overview of concepts from dynamical systems. 
Let $(X,d)$ be a nonempty compact metric space, and 
let $f\colon X\to X$ be a continuous map, that is, let $(X,f)$ be a dynamical system. 
We will be interested in the long-term behavior, with special emphasis on the limiting behavior. To this aim, we introduce a few ideas. 
A fixed point $x$ of a dynamical system $(X,f)$ is called
\begin{itemize}
    \item attracting, if there is an open neighborhood $U\subset X$ of $x$ such that $f(U)\subset U$ and for every $y\in U$, we have $\lim_{m\to\infty}f^m(y)=x$, where $f^m$ is a composition of the map $f$ with itself $m$-times;
    \item repelling, if there is an open neighborhood $U\subset X$ of $x$ such that for every $y\in U$, $y\neq x$, there exists $m\in\mathbb{N}$ such that $f^m(y)\notin U$.
\end{itemize}
An orbit $(f^m(x))$ is called periodic of period $n$ if $f^n(x)=x$. The smallest such $n$ is called the period of $x$. The periodic orbit is called attracting, if $x$ is an attracting fixed point of $(X,f^n)$, and repelling, if $x$ is a repelling fixed point of $(X,f^n)$.

Generally, we cannot assume that trajectories of points will be attracted by a periodic orbit. Thus, we consider $\omega$-limit sets. For any $x\in X$ its $\omega$-limit set $\omega(x,f)$ consists of points $y\in X$ for which there exists a strictly increasing sequence of natural numbers $\{m_k\}$ such that $f^{m_k}(x)\to y$ as $k\to\infty$. Thus, \[\omega(x,f)=\bigcap_{m\in\mathbb{N}}\overline{\{f^k(x):k>m\}}.\]
Obviously $\omega$-limit set of a point belonging to periodic orbit is the periodic orbit itself.

To connect the dynamical systems theory with the ergodic theory, we consider the set $M(X)$ of all probability Borel measures on $X$ (here for a measure $\mu$ "probability" means that $\mu(X)=1$). This space has a subspace $M_{inv}(X)$ consisting of all invariant measures from $M(X)$, that is $\mu(A)=\mu(f^{-1}(A))$ for every measurable set $A\subset X$. Another point of view on invariance is to consider an operator $f_*:M(X)\to M(X)$, defined by $f_*(\mu)(A)=\mu(f^{-1}(A))$ for every measurable set $A\subset X$. Then $\mu$ is invariant if and only if it is a fixed point of $f_*$.

The set $M(X)$ is equipped with the weak-* topology. In this topology, a sequence $(\mu_m)$ converges to $\mu$ if and only if for every observable (continuous function) $\varphi\colon X\to\mathbb{R}$ the sequence $\big(\int\varphi \,d\mu_m\big)$ converges to $\int\varphi\,d\mu$. In this topology,
$M(X)$ is compact.
Interestingly, if a sequence $\big(\frac1{N_i}\sum_{k=0}^{N_i-1}f^k_*(\mu)\big)$, for some $N_i\to\infty$, converges, then the limit is automatically invariant.

An important class of measures are measures absolutely continuous with respect to the Lebesgue measure. These are measures $\mu$ such that if $\lambda$ is the Lebesgue measure then $\lambda(A)=0$ implies $\mu(A)=0$. They are important because the Lebesgue measure itself gives us a natural way of measuring how large a set is. Thus, more or less the same can be said about the absolutely continuous measures. A lot of work has been devoted to proving that in certain situations a dynamical system has an absolutely continuous invariant measure.

A measure $\mu\in M_{inv}(X)$ is called ergodic if for any measurable set $A$ satisfying $A=f^{-1}(A)$ we have $\mu(A)=0$ or $\mu(A)=1$. Those measures are extreme points of $M_{inv}(X)$. Their importance is to a great deal due to the Birkhoff Ergodic Theorem, which claims that if $\mu$ is ergodic then the limit of time averages of any $L^1$ function $\varphi\colon X\to\mathbb{R}$ exists almost everywhere and is equal to the space average of $\varphi$, $\int\varphi\;d\mu$.

The next important notion we will be using, is the negative Schwarzian derivative. Let $X$ be a closed interval. The Schwarzian derivative of $f\colon X\to X$ is given for the formula $Sf=\frac{f'''}{f'}-\frac32\big(\frac{f''}{f'}\big)^2$. If $Sf<0$ (except at the critical points) then also $Sf^k<0$ for any $k>0$. Now, if $Sf<0$ and $x$ is an attracting periodic orbit, there must be a critical point of $f$ (or an endpoint of $X$), whose trajectory is attracted to the orbit of $x$. This is an important property, in particular bounding the number of attracting periodic orbits by the number of critical points.

We will also use the notion of topological conjugacy. We will say that dynamical systems $(X,f)$ and $(Y,g)$ are (topologically) conjugate if there exists a homeomorphism $\psi:X\to Y$ such that $\psi\circ f=g\circ\psi$. This is a kind of isomorphism in the category of dynamical systems. In particular, many of the properties of the system stay the same (with some obvious minor changes) if we apply a conjugacy. For example, if $P$ is a periodic orbit in $(X,f)$ then $\psi(P)$ is a periodic orbit of $(Y,g)$ with the same properties (period, attracting, repelling, etc.). If $\mu$ is an invariant measure for $(X,f)$, then $\psi_*(\mu)$ is an invariant measure for $(Y,g)$. If one measure is ergodic, so is the other one. However, in order to preserve absolute continuity, we need $\psi$ to be smooth. Observe that conjugacy usually does not preserve the center of mass of an invariant measure (for example, when $\psi(x)=x+1$, then the center of mass will move by 1 to the right).

\section{Natural measures}
\label{sec:natural}
For a given dynamical system, such as a map $f$ from a space $X$ to itself, a fundamental question is how trajectories of points in $X$ evolve over time. While in simple cases, they may converge to a fixed point or a periodic orbit, in more complex systems, the behavior can be chaotic.

To understand such complex behavior, we often resort to statistical methods. Instead of tracking individual points, we consider the distribution of points over time. This leads us to the concept of invariant measures. An invariant measure $\mu$ is a probability distribution on $X$ that remains unchanged under the action of the map $f$.

A particularly important class of invariant measures are natural measures. These measures capture the long-term statistical behavior of the system, especially for initial conditions that are typical with respect to a reference measure (often the Lebesgue measure). Natural measures are crucial in understanding the system's sensitivity to initial conditions, predicting the frequency of events, and computing averages of observables.

By studying natural measures, we can gain insight into the underlying dynamics of complex systems, from simple maps to high-dimensional chaotic systems.

The basic question is the behavior of its orbits $(f^m((x))$, as the time $m$ goes to infinity. 
The simplest situation occurs when there
exists a globally attracting point $p\in X$, that is, a point such
that $\lim_{m\to\infty}f^m(x)=p$ for every $x\in X$. However, in most
cases we cannot count for such luck. Therefore, we have to lower our
expectations. The first change will be to replace every $x\in X$ by
every $x$ from a set of positive measure for some important
probability measure on $X$ (we will call it a \emph{reference measure}).
The second change will be to work with probability measures on $X$
instead of points of $X$. The map $f$ will be replaced by the induced
map $f_*$, defined in the preceding section. The third change will be
to replace the orbit $(f_*(\mu))$ by the time averages
$\frac1m\sum_{k=0}^{m-1}f_*^k(\mu)$.


Let us say more about the proposed changes. The reference measure we talk about in the first change is usually the Lebesgue (or, on manifolds, Riemannian) measure. In the study of dynamical systems, Lebesgue measure naturally corresponds to our intuitive notion of volume and experimental observability; an invariant measure is considered physically relevant precisely when it governs the time-averages for a set of initial conditions with positive Lebesgue measure \cite{young2002srb}. Therefore, a set of positive measure can be treated as a large, observable one.

The second change comes from a natural idea that perhaps we should look
at our system from a more global perspective than the pointwise one.
What cannot be seen when looking at one orbit maybe can be seen when
looking at the orbit of a measure, which takes into account many
individual orbits. Of course, we have to say what the map on the space
$\calm(X)$ (of all probability measures on $X$) is and what is the
topology in this space. That is why we introduce $f_*$ and weak-* topology introduced in the previous section.

The replacement from the third change makes sense since we are
observing not only the position of a point after applying $m$
iterates, but the whole orbit from time 0 to $m-1$. Thus, when for
instance, we have an attracting periodic point instead of an
attracting fixed point, the limit of $f^m(x)$ does not exist, while
the limit of time averages does.

Now we have to look whether with our changes we get something similar
to the globally attracting fixed point. And indeed, we do. Blank and
Bunimovich~\cite{blank2003multicomponent} call a measure $\muf$ \emph{natural} if there
exists an open nonempty subset $U\subset X$ (called the \emph{basin of
  attraction} for the measure $\muf$) such that for any measure
$\mu\in\calm(X)$ absolutely continuous with respect to the reference
measure and having its support in $U$ (that is, $\mu(U)=1$) we
have $\lim_{m\to\infty}\frac1m\sum_{k=0}^{m-1}f_*^k(\mu)=\muf$. As
shown in~\cite{misiurewicz2005ergodic}, if $\muf$ is ergodic, then it suffices to assume
this convergence for one such measure $\mu$, and automatically it
holds for all such measures. Thus, with our changes, a natural measure
plays the role of a globally attracting fixed point among measures
absolutely continuous with respect to the reference measure.\footnote{Moreover, it has been shown that in practice numerical simulations often display the absolutely continuous invariant measures with respect to Lebesgue measure \cite{gora1988computers}.}

Let us mention that natural measures are always invariant. That is,
$\muf(A)=\muf(f^{-1}(A))$ for every measurable set $A\subset X$. In
other words, $\muf$ is a fixed point of $f_*$.

In one-dimensional smooth dynamics one of the central problems is
finding invariant measures, absolutely continuous with respect to the
Lebesgue measure. They are usually ergodic and are natural
measures for the Lebesgue measure as the reference measure. Their
existence allows us to use powerful machinery of the Ergodic Theory;
in particular to prove various statistical properties of the system.

\section{Theoretical results}\label{Th-results}
\label{sec:theo}
 In this section we will show some facts on averages of observables, an invariant measure and its support.\footnote{In Section \ref{Th-results} we consider always only one measure
(there is no reference measure), so the term {\it almost everywhere} is not ambiguous.} 
We will make the following assumption.

\begin{assumption}\label{as1}
Let $(X,f)$ be a dynamical system, that is $X$ is nonempty and compact and $f\colon X \to X$ is a continuous map. Let $O\colon X\to\R$ be an observable (a continuous function).
\end{assumption}

In this section we show that convergence of the time averages of observables determine 
the average observation of the system with respect to a measure invariant for $(X,f)$. Then we show that if there exists a globally attracting periodic orbit of the system, then the time-averaged observations converge. Finally, we describe the structure of the support of an absolutely continuous measure for specific class of maps, which contains all game dynamics, which we will discuss in Section \ref{sec:MWU}.

\begin{theorem}\label{thm1}
Let Assumption \ref{as1} hold and let $x\in X$. If
\begin{equation}\label{eq1}
\lim_{i\to\infty}\frac1{N_i}\sum_{k=0}^{N_i-1}O(f^k(x))=y
\end{equation}
for some $N_i\to\infty$ and $y\in \R$, then there exists a probability
measure $\mu$ on the $\omega$-limit $\omega(x,f)$ of $x$, such that
$\int O\;d\mu=y$. In particular, this measure $\mu$ is invariant for $(X,f)$.
\end{theorem}

\begin{proof}
Since the space of all probability measures on $X$ with the weak-*
topology is compact, from the sequence
$\frac1{N_i}\sum_{k=0}^{N_i-1}f^k_*(\delta_x)$ we can
choose a convergent subsequence. In other words, we can assume that the
sequence
$\lim_{i\to\infty}\frac1{N_i}\sum_{k=0}^{N_i-1}f^k_*(\delta_x)$
converges to some measure $\mu$. We have
\[\hspace{-0.1in}\int O\;d\mu=\lim_{i\to\infty}\int O\;d\Big(\frac1{N_i}\sum_{k=0}^{N_i-1}
f^k_*(\delta_x)\Big)=\lim_{i\to\infty}\Big(\frac1{N_i}\sum_{k=0}^{N_i-1}
\int O\;d(f^k_*(\delta_x)\Big)\\
=\lim_{i\to\infty}\frac1{N_i}\sum_{k=0}^{N_i-1}O(f^k(x))=y.\]

It remains to prove that $\mu$ lives on $\omega(x,f)$. For this, we
repeat the above construction on
\begin{equation}\label{eq2}
\omega(x,f)\cup\{f^k(x):k>I\}
\end{equation}
instead of $X$, where $I$ is a positive integer. We get the same measure $\mu$. Since the intersection
of those sets for every positive integer $I$ is $\omega(x,f)$, the
measure $\mu$ lives on $\omega(x,f)$.
\end{proof}

Theorem \ref{thm1} guarantees two facts: (i) that a probability measure exists; (ii) that the average observation with respect to the probability measure aligns with the time-averaged observation of the evolution of the system. In a rather interesting way, this last property is analogous to the Birkhoff's Theorem for ergodic systems, although our systems are not necessarily ergodic.
We will use both of these results in the proof of the next theorem.

\begin{theorem}\label{thm2}
Let Assumption \ref{as1} hold. Let $\mu$ be a probability invariant measure for $(X,f)$, where $X$ is a convex subset of a linear space, and $p$ be a unique interior fixed point of the map $f$. Assume that the time-average trajectory converges to the fixed point $p$ almost everywhere, that is 
\begin{equation}\label{thm2:eq-1}
 \lim\limits_{m \to \infty} \frac{1}{m} \sum_{k=0}^{m-1} f^k(x) = p
\end{equation}
holds for all $x$ except for a set of measure zero.
\begin{enumerate}
\item If $O$ is linear/affine, then
$ \int O\: \dmu = O(p)$.

\item If $O$ is convex (concave), then
$\int O\: \dmu \underset{(\leqslant)}{\geqslant} O(p)$.
\end{enumerate}

\end{theorem}

\begin{proof}
We begin with the proof of the assertion 1.
Since $O$ is linear/affine, we get from \eqref{thm2:eq-1} that
\begin{equation}\label{thm2:eq0}
 \lim\limits_{m \to \infty} \frac{1}{m} \sum_{k=0}^{m-1} O(f^k(x)) = O(p).
\end{equation}

If the fixed point $p$ is globally attracting, then $\delta_p$ is the only invariant measure. Thus, by Theorem \ref{thm1} and by the equation \eqref{thm2:eq0} we get
\[
 \int O\: \dd_p = \lim\limits_{m \to \infty} \frac{1}{m} \sum_{k=0}^{m-1} O(f^k(x)) = O(p).
\]

We now consider the case when the fixed point $p$ is not globally attracting. First, assume that $\mu$ is an ergodic measure.
Recall that Birkhoff's theorem holds for all $x$ except for a set of measure zero. So take $x$ from the intersection of the sets where \eqref{thm2:eq-1} holds and Birkhoff's theorem holds (this intersection is nonempty). Then by using Birkhoff's theorem and the equation \eqref{thm2:eq0} we obtain
\begin{equation}\label{thm2:eq1}
 \int O\: \dmu = \lim\limits_{m \to \infty} \frac{1}{m} \sum_{k=0}^{m-1} O(f^k(x)) = O(p).
\end{equation}
Second, assume that $\mu$ is not an ergodic measure. Denote the set of probability measures invariant for $(X,f)$ by $\mathcal{M}_f$. 
The space $X$ is compact, thus  by Theorem \ref{thm1} the set $\mathcal{M}_f$ is nonempty, and by Banach-Alaoglu Theorem it is compact in the weak-* topology. Moreover, as a convex combination of invariant measures is an invariant measure, it is convex.
In addition, for the set $\mathcal{M}_f$ we know that the extreme points of $\mathcal{M}_f$ coincide with ergodic measures (see e.g., Theorem 4.4 in \cite{EWbook2011}).
In addition, by Krein-Milman Theorem convex combinations of extreme points are dense in $\mathcal{M}_f$.

If $\mu$ is a convex combination of ergodic measures, that is $\mu = \alpha_1 \nu_1 + \ldots + \alpha_l \nu_l$, where $\alpha_1, \ldots, \alpha_l \in [0,1]$, $\sum_{i=1}^l \alpha_i = 1$ and $\nu_1, \ldots, \nu_l$ are ergodic, then we get by \eqref{thm2:eq1} that
\begin{equation}\label{thm2:eq2}
 \int O\: \dmu = \sum_{i=1}^l \alpha_i \int O\: \dnu_i = \sum_{i=1}^l \alpha_i O(p) = O(p).
\end{equation}
On the other hand, if $\mu$ is not a convex combination of ergodic measures, then there exists a sequence of convex combination of ergodic measures $(\lambda_i)_{i \in \N}$ convergent to $\mu$. For each element of the sequence $\{\lambda_i\}_{i \in \N}$ the equation \eqref{thm2:eq2} holds, therefore the sequence $\int O\: \dl_i$ is constant. As a result,
\[
 \int O\: \dmu = \lim_{i \to \infty} \int O\: \dl_i = \lim_{i \to \infty} O(p) = O(p).
\]

We proceed to the proof of the assertion 2.
From Jensen's inequality and \eqref{thm2:eq-1} we obtain
\[
 \lim\limits_{m \to \infty} \frac{1}{m} \sum_{k=0}^{m-1} O(f^k(x)) \geqslant O \left( \lim\limits_{m \to \infty} \frac{1}{m} \sum_{k=0}^{m-1} f^k(x) \right) = O(p),
\]
when $O$ is convex, and
\[
 \lim\limits_{m \to \infty} \frac{1}{m} \sum_{k=0}^{m-1} O(f^k(x)) \leqslant O \left( \lim\limits_{m \to \infty} \frac{1}{m} \sum_{k=0}^{m-1} f^k(x) \right) = O(p),
\]
when $O$ is concave. The rest of the proof follows the steps of the proof of the assertion 1.
\end{proof}

In other words, Theorem \ref{thm2} shows that for the special case of linear observables the space-average observation has a unique value that is independent of the choice of an invariant measure.

Next, we will provide a useful definition in terms of making
measurements of the system state.

For a given initial state $x$ we define \emph{the average of the
first $m>0$ observations} by
\begin{equation}
\hat{O}_m(x) := \frac{1}{m} \sum_{k=0}^{m-1} O(f^{k}(x)).
\end{equation}

and \emph{the time-averaged observation} as
\[
\hat{O}(x) := \lim_{m \rightarrow \infty} \hat{O}_m(x) =
\lim_{m \rightarrow \infty} \frac{1}{m} \sum_{k=0}^{m-1} O(f^{k}(x))
\]
when the appropriate limit exists.

\begin{proposition}\label{averages2}
Let Assumption \ref{as1} hold. Assume that the dynamical system $(X,f)$ has an attracting periodic orbit of period $n > 0$ that attracts almost all initial states.
Then for almost all points $x \in X$
\[
\lim_{i\rightarrow \infty} \hat{O}(f^i(x)) = \hat{O}(z),
\]
where $z$ is an arbitrary point of the attracting periodic orbit.
\end{proposition}

\begin{proof}
We are going to show the proof in two steps.

Step 1. We will show that if $n$ divides $m$, then for almost all initializations, the limit of the $m$-th iterate time-averaged observations for any continuous function $O$ is captured by its time average over the attracting periodic orbit, or in other words, that
for almost all points $x \in X$
\begin{equation}\label{thm_av-eq1}
\lim_{i \rightarrow \infty} \hat{O}_m(f^i(x)) =  \hat{O}_m(z).
\end{equation}
Clearly, if $n$ divides $m$ then for every choice of the element $z$
of the periodic orbit the averages $\hat{O}_m(z)$ are equal to each
other. Let $x_i = f^i(x)$ be a trajectory attracted to the orbit of
$z$. It means that the sequence $(x_i)$ can be split into disjoint
subsequences $(x_{i_j})_{j \in \{1,2,\ldots,n\}}$ such that each
$(x_{i_j})$ converges to an element of the attracting periodic orbit
of $z$. Because the observable $O$ is continuous, we have that the
subsequences $\hat{O}_m(x_{i_j})$, for $j \in \{1,2,\ldots,n\}$,
converge to the same limit $\hat{O}_m(z)$. 
Thus, $\hat{O}_m(x_{i})$ also converges to $\hat{O}_m(z)$.

Step 2. Let $\varepsilon > 0$ and let $x_i = f^i(x)$ be a trajectory attracted to the orbit of
$z$. We are going to show that
\[
 | \hat{O}(z) - \hat{O}(x_i) | < \varepsilon
\]
for $i$ sufficiently large.

Choose $m>0$ such that $m=0 \mod (n)$.  Then because $z$ is an element of the periodic orbit of period $n$ we have that $\hat{O}(z) = \hat{O}_m(z)$.
Second, from the definition of the time-average $\hat{O}$ we get that
\begin{equation}\label{thm_av-eq2}
 | \hat{O}_m (x_i) - \hat{O} (x_i) | < \frac{\varepsilon}{2}
\end{equation}
for $m$ sufficiently large. Therefore, by combining \eqref{thm_av-eq1}, \eqref{thm_av-eq2} and the fact that $\hat{O}(z) = \hat{O}_m(z)$ we obtain
\begin{align*}
 \big| \hat{O} (z) - \hat{O} (x_i) \big| &= \big| \hat{O} (z) - \hat{O}_m (x_i) + \hat{O}_m (x_i) - \hat{O} (x_i) \big| \\
 &\leqslant \big| \hat{O}_m (z) - \hat{O}_m (x_i) \big| +  \big| \hat{O}_m (x_i) - \hat{O} (x_i) \big| < \frac{\varepsilon}{2} + \frac{\varepsilon}{2} = \varepsilon
\end{align*}
for $i$ and $m$ sufficiently large.
\end{proof}

The result of Proposition \ref{averages2} implies that for almost all initializations, the time-averaged
observations for any continuous function $O$ is captured by its time average over the attracting periodic orbit.

Now, if we look for an absolutely continuous invariant measure it is useful to guarantee that the support of this measure is simple. This happens for maps with "good" properties:

\begin{proposition}
\label{lem:supcycle}
    Let $f\colon [0,1]\to [0,1]$ be a smooth multimodal map with negative Schwarzian derivative and non-degenerated critical points. Let $\mu$ be an absolutely continuous (ergodic) measure. Then the support of $\mu$ is a cycle of intervals.
\end{proposition}
\begin{proof}
  By \cite{blokh1990decomposition} we know that if the support of $\mu$ is not a cycle of intervals, then the only possibility is that the measure lives on a wild attractor (which is the $\omega$-limit set of a critical point that is minimal). Then by Proposition 24 of \cite{bruin2010lebesgue} this wild attractor has Lebesgue measure zero, which leads to a contradiction.
\end{proof}

\section{MWU driven game dynamics}
\label{sec:MWU}
In this article we aim to introduce natural measures in the game-theoretic perspective. Thus, we will describe natural measures for the game dynamics introduced by the multiplicative weights update algorithm. We will also show that game dynamics arising from this simple algorithm depicts all theoretically possible behaviors of one-dimensional dynamical systems including existence of natural invariant measures. 

\subsection{Settings}
Let us consider a two-strategy congestion game ~\cite{rosenthal73} with a population (continuum) of players (agents) of size $N>0$, where all use the multiplicative weights update learning algorithm to update their strategies \cite{Arora05themultiplicative}. Each of the players controls an infinitesimal small fraction of the flow. The cost of each path (link, route, or strategy) will be assumed to be proportional to the load. By denoting $c(j)$ the cost of selecting the strategy number $j$ (when $x$ fraction of the agents choose the first strategy), if the coefficients of proportionality are $\alpha,\beta>0$, we obtain
\begin{equation}\label{cost}
c(1)=\alpha N x, \hspace{50pt} c(2)=\beta N (1-x).
\end{equation}

Without loss of generality we assume that $\alpha+\beta=1$. Therefore, the values $\alpha$ and $\beta=1-\alpha$  indicate how different the path costs are from each other. 
At time $n+1$, the players know the cost of the strategies at
time $n$ (equivalently, the probabilities $(x_n, 1-x_n)$) and update
their choices. Since we have a continuum of agents, 
the realized flow (split) is accurately described by the probabilities $(x_n, 1-x_n)$.
The algorithm for updating the probabilities that we focus on is the
\emph{multiplicative weights update} (MWU) the ubiquitous learning algorithm widely employed in Machine Learning, optimization, and game theory (also known as Normalized Exponentiated Gradient, Hedge algorithm) \cite{Arora05themultiplicative,auer1995gambling}:

\begin{equation}\label{mwu}
x_{n+1}=\frac{x_n(1-\eps)^{c(1)}}{x_n(1-\eps)^{c(1)}+
(1-x_n)(1-\eps)^{c(2)}}
           =\frac{x_n}{x_n+
(1-x_n)(1-\eps)^{c(2)-c(1)}},
\end{equation}
with $\eps\in(0,1)$, being treated as the common learning rate of
all players.

By introducing the new variables $a=N \ln\left(\frac1{1-\eps}\right)$, and $b=\beta$ we obtain game dynamics driven by the one-dimensional map:
\begin{equation}\label{map}
\begin{aligned}
f_{MW}(x)&=\frac{x}{x+(1-x)\exp(a(x-b))}.\\
\end{aligned}
\tag{MW}
\end{equation}
This map has already been studied and Li-Yorke chaos was shown in \cite{chotibut2021family,CFMP2019}.
In particular, this map is increasing when $a\leq 4$. When $a>4$ it has two critical points given by
\[c_l=\frac 12-\sqrt{\frac 14-\frac 1a},\;\; c_r=\frac 12+\sqrt{\frac 14-\frac 1a},\]
and negative Schwarzian derivative \cite{chotibut2021family}. Fixed points of $f_{MW}$ are $0,1$ and $b$. While $0$ and $1$ are always repelling, the interior fixed point (being also Nash equilibrium of the game) is globally attracting as long as $a<\frac{2}{b(1-b)}$ (see Theorem 3.8 from \cite{chotibut2021family}). But when $a>\frac{2}{b(1-b)}$ the interior fixed point become repelling, thus excluding the possibility of the convergence of game dynamics to the equilibrium and undermining any equilibrium analysis. 

To study statistical properties of the game dynamics we will look at the conjugate map to \eqref{map}.
introduced by taking $y=\frac 1a \ln\frac{x}{1-x}$ and \[F(y)=\frac 1a
\ln \frac{f_{MW}(x)}{1-f_{MW}(x)}.\] Then the family  $f_{MW}$ on $(0,1)$ is conjugate to the map  
$F\colon \R\to\R$, given by
\begin{equation}\label{e2}
F(y)=y+b-\frac1{e^{-ay}+1},
\end{equation}
parametrized by $a>0$ and $b\in (0,1)$.\footnote{This map is known in the literature, see \cite{eirola1996chaotic,kryzhevich2021bistability,bielawski2025interval}.} 

By the conjugacy argument, this map has two critical points for $a>4$, with interesting dynamics for $a>\frac{2}{b(1-b)}$.
We denote the corresponding left critical point, right critical point and the fixed point of the map $F$ by $y_l$, $y_r$ and $y_f$ respectively.
Let us recall here result from \cite{bielawski2025interval} (Lemmas 3.1 and 3.2 combined).

\begin{lemma}\label{lemma-recall}
Let $b = k/n$, where $k,n$ are coprime natural numbers and $k<n$. Then the map given by \eqref{e2} has an attracting periodic orbit of period $n$ for sufficiently large values of $a$.
Moreover, the attracting periodic orbit attracts trajectories of Lebesgue almost all points from $\R$.
\end{lemma}

Since the existence of an attracting periodic orbit is a property of a dynamical system that caries through the topological conjugacy, we have that the result of Lemma \ref{lemma-recall} holds for the MWU dynamics as well.

\subsection{Natural measures for MWU dynamics}
\label{sec:theoMWU}
By \cite{chotibut2021family} we know that the time-average trajectory of the MWU dynamics converges to the interior fixed point (the Nash equilibrium of the game) for every starting point $x \in (0,1)$. Thus, the assumptions of Theorem \ref{thm2} are satisfied. Moreover, if the observable $O$ is linear/affine, then the assumptions of Theorem \ref{thm1} are also satisfied. As a result, we have the following observation.

\begin{corollary}
Let $O:[0,1]\to\R$ be continuous function (an observable) and $x \in (0,1)$.
 \begin{enumerate}
\item If $O$ is linear/affine, then there exists a probability measure $\mu$ on $\omega(x,f_{MW})$ and
 $\int O\: \dmu = O(b)$.

\item If $\mu$ is a probability measure on $[0,1]$ invariant for $([0,1],f_{MW})$ such that $\mu(\{0,1\})=0$ and $O$ is convex (concave), then
\[
 \int O\: \dmu \underset{(\leqslant)}{\geqslant} O(b).
\]
\end{enumerate}
\end{corollary}

In the special case when the observable is the cost function, we get for the congestion game that the observable $O$ is linear. Then we have the following fact:

\begin{corollary}
For any $x \in (0,1)$ there exists a probability measure $\mu$ on the $\omega$-limit set.  Moreover, the space-average cost (with respect to the measure $\mu$) is equal to the cost at the Nash equilibrium.
\end{corollary}

By applying Lemma \ref{lemma-recall} and conjugacy argument we see that the assumptions of Proposition \ref{averages2} are satisfied for $f_{MW}$, and therefore we have the following observation.

\begin{corollary}
Let $b = k/n$, where $k,n$ are coprime natural numbers and $k<n$.
Let $O\colon [0,1]\to\R$ be a continuous function (an observable). If the parameter $a$ is large enough, then there exists an attracting periodic orbit of $f_{MW}$ that attracts trajectories of Lebesgue almost all points from $[0,1]$. Moreover, for almost all points $x \in [0,1]$
\[
\lim_{i\rightarrow \infty} \hat{O}(f_{MW}^i(x)) = \hat{O}(z),
\]
where $z$ is an arbitrary point of the attracting periodic orbit.
\end{corollary}

Finally, as the map $f_{MW}$ is bimodal with two critical points and has negative Schwarzian derivative when $a>4$, the assumptions of Proposition \ref{lem:supcycle} are satisfied, and thus we have the following observation.

\begin{corollary}
Let $\mu$ be an absolutely continuous (ergodic) measure on $[0,1]$. If $a>4$ then the support of $\mu$ is a cycle of intervals.
\end{corollary}

In~\cite{lyubich2002}, Lyubich proved that for Lebesgue almost all
values of the parameter, the quadratic map of an interval has either
an attracting periodic orbit or an absolutely continuous invariant
measure. By absolutely continuous we mean absolutely continuous with
respect to the Lebesgue measure. We also assume that the measure of
the whole space is 1. There are reasons \cite{bruin2010lebesgue,blokh1990decomposition} to suspect that similar
properties hold for other ``natural'' families of maps with negative
Schwarzian derivative and nondegenerate critical points, so we will
concentrate on those two types of behavior.

For $a>4$ the map $F$ has a negative Schwarzian derivative and two critical
points, both nondegenerate. This means that the behavior of the
trajectories is ruled by the behavior of the trajectories of the
critical points. In particular, if there is an attracting periodic
orbit, then there is a critical point in the immediate basin of
attraction of this orbit (that is, there is an interval joining
this critical point with some point of our periodic orbit, such
that the trajectory of every point of this interval converges to
our periodic orbit) \cite{de2012one}. 
A more careful treatment is needed when one looks for an absolutely continuous invariant measure. Nevertheless, by Proposition \ref{lem:supcycle}, in our case this measure has to live at a cycle of intervals.
Thus, if there is an absolutely continuous invariant measure, its support (the smallest
closed set of full measure) is an invariant cycle of intervals,
and there is a critical point in this support. In this article we will focus on possible behavior of such cycles for game dynamics driven by \eqref{map}. 

Observe that if there is an attracting periodic orbit, then its
immediate basin of attraction is an invariant cycle of intervals.
Remember that when we speak of a cycle of intervals, we include the
possibility that the length of this cycle is 1. However, we exclude
the whole interval on which our map is defined, unless strictly inside
it there is no other invariant cycle of intervals.

In such a way we get the following possible behaviors of our maps.
\begin{enumerate}
\item There is one invariant cycle of intervals (Section \ref{sec:oneinvariant}).
  \begin{enumerate}
  \item There is an attracting periodic orbit (Section \ref{sec:gapo}).
    \begin{enumerate}
    \item Both critical points are in the immediate basin of
      attraction of the periodic orbit.
    \item Only one critical point is in the immediate basin of
      attraction of the periodic orbit.
    \end{enumerate}
  \item There is an absolutely continuous invariant measure.
    \begin{enumerate}
    \item Both critical points are in the support of this measure (Section \ref{sec:acim}).
    \item Only one critical point is in the support of this measure (Section \ref{sec:one}).
    \end{enumerate}
  \end{enumerate}
\item There are two disjoint invariant cycles of intervals (Section \ref{sec:2cycles}).
  \begin{enumerate}
  \item There are two attracting periodic orbits (Section \ref{sec:2apo}).
  \item There is one attracting periodic orbit and one absolutely
    continuous invariant measure (Section \ref{sec:2apo}).
  \item There are two absolutely continuous invariant measures (Section \ref{sec:2measures}).
  \end{enumerate}
\end{enumerate}


In the following section we will show that the game dynamics introduced by multiplicative weights update algorithm in simple congestion games demonstrates all of these behaviors. Thus, possibilities for this game dynamics range from convergence to periodic orbit of all initial points (excluding a set of measure zero) to more sophisticated behavior with a mix of absolutely continuous invariant measures and attracting periodic orbit.  

\subsection{One invariant cycle}
\label{sec:oneinvariant}
\subsubsection{(Globally) attracting periodic orbit}
\label{sec:gapo}
We begin with the simplest behavior described in 1.a.
By Lemma \ref{lemma-recall}, when $b=\frac kn$ and  $a$ is sufficiently large, there exists an attracting periodic orbit of period $n$ such that the trajectories of Lebesgue almost all initializations of $f_{MW}$ are attracted to it. Then one observes only a unique attracting periodic orbit (case 1.a). Nevertheless, taking $a$ (and thus the learning rate) very large is not always the proper choice in applications. What happens for more reasonable choices? Are there other possibilities? Moreover, which of behaviors from 1.a one can really see? 
We show in the $(a,b)$-plane in Figure~\ref{f2} that there exist both ranges of parameters for which both critical points lay in the immediate basin of the unique attracting periodic orbit, and ranges of parameters for which only one of them is attracted to such orbit.\footnote{In addition, we detect values for which one can spot two attracting periodic orbits, the case which will be described in details in Section \ref{sec:2apo}.} 
\begin{figure}[h!]
\begin{center}
\includegraphics[width=0.6\textwidth]{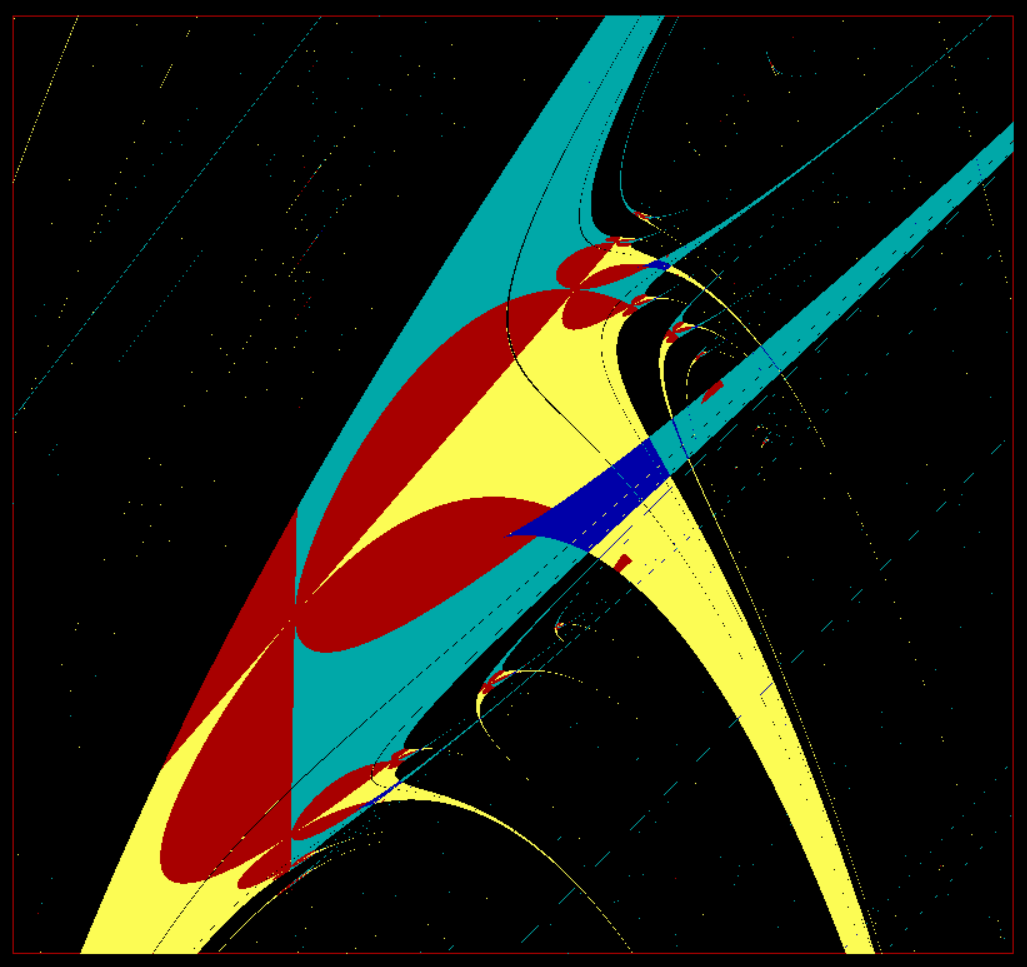}
\caption{Various periodic orbits for various parameter
  values.
  The horizontal axis is $a$, from $32.02$ to $32.1$, the vertical axis is
$b$, from $0.2942$ to $0.2942495$. The color of the pixel depends on
the behavior of the trajectories of critical points for the values of
$a,b$ corresponding to this pixel. {\color{red}Red} means that both
critical points are in the immediate basin of attraction of the same
periodic orbit (case 1.a.i). {\color{yellow}Yellow} means that the
left critical point is in the immediate basin of attraction of a
periodic orbit, but the right one is not (case 1.a.ii).
{\color{cyan}Cyan} means that the right critical point is in the
immediate basin of attraction of a periodic orbit, but the left one is
not (also case 1.a.ii). {\color{blue}Blue} means that both critical
points are in the immediate basins of attraction of different periodic
orbits (case 2.a).
  }\label{f2}
\end{center}
\end{figure}

\subsubsection{Absolutely continuous invariant measure}
\label{sec:acim}
Here we want to find at least one value of parameters $a,b$ for which
we are sure that there is a unique absolutely continuous invariant
measure and no attracting periodic orbits.  This will be case 1.b.i. But showing existence of an absolutely continuous invariant measure in numerical experiments is not so evident.
For this, we look for the values of $a,b$, for which the second image
of the right critical point $y_r$ is equal to the left critical point
$y_l$, and the third image of $y_l$ is equal to the fixed point $y_f$.
By~\cite{Mis_acm}, this implies the existence of such measure.
Figure~\ref{f3} suggests that this occurs for $a\approx31.047$ and
$b\approx0.374875$.

The measure is unique, since otherwise the right critical point would
be in the support of one measure and the left one in the support of
the other measure. Those supports have to be disjoint and invariant,
but the second image of the right critical point is equal to the left
critical point, and we would get a contradiction.

\begin{figure}[h!]
\begin{center}
\includegraphics[width=60truemm]{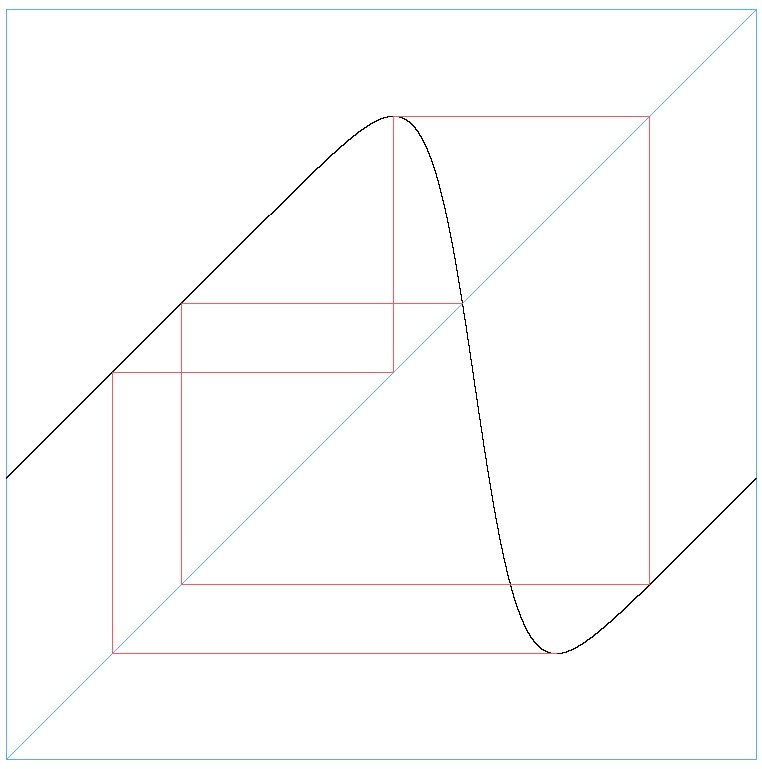}
\caption{$F^2(y_r)=y_l$ and $F^3(y_l)=y_f$.}\label{f3}
\end{center}
\end{figure}

However, this is a very week numerical evidence, since small changes
in parameters can change the picture in an unpredictable way. To get a
stronger evidence, consider the rectangle in the $(a,b)$-plane with
sides at $a=27.5$, $a=40$, $b=0.25$, $b=0.5$ (see Figures
\ref{f4a}--\ref{f4d}). For the four sides of the rectangle, we draw
the graphs of $F^3(y_l)-y_f$ (black) and $F^2(y_r)-y_l$ (green). The magenta line is at the level 0.

\begin{figure}[h!]
\centering
\begin{subfigure}{0.45\textwidth}
\begin{center}
\includegraphics[width=0.75\textwidth]{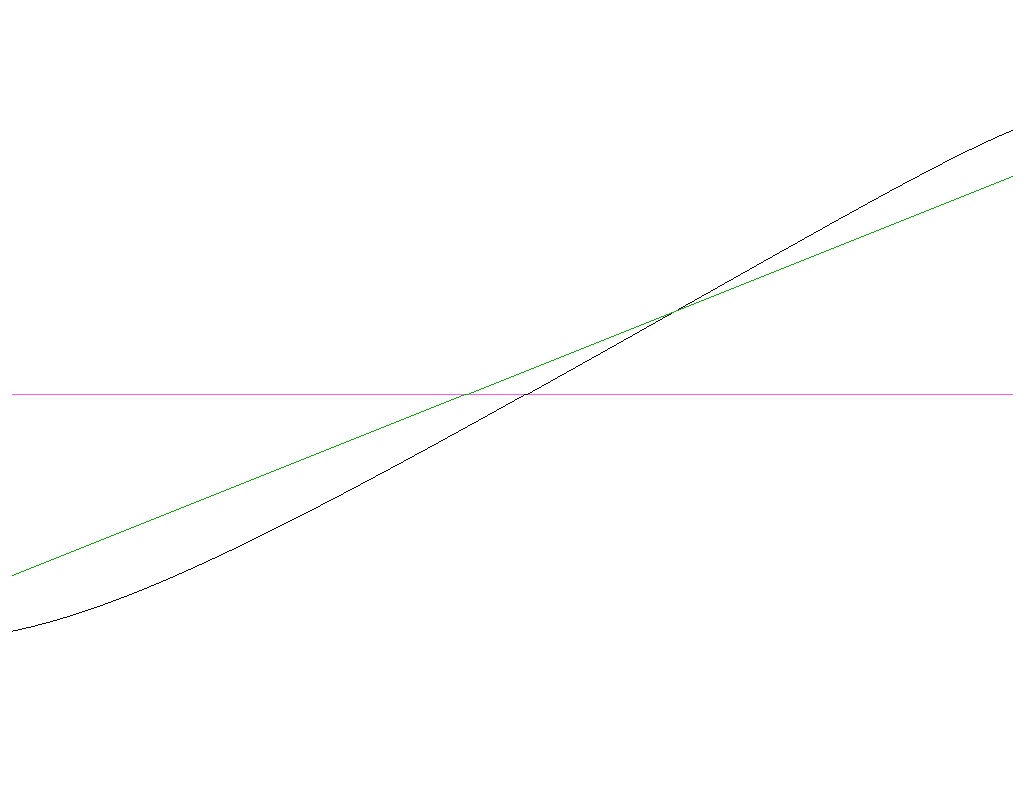}
\caption{The side $a=27.5$.}\label{f4a}
\end{center}
\end{subfigure}
\begin{subfigure}{0.45\textwidth}
\begin{center}
\includegraphics[width=0.75\textwidth]{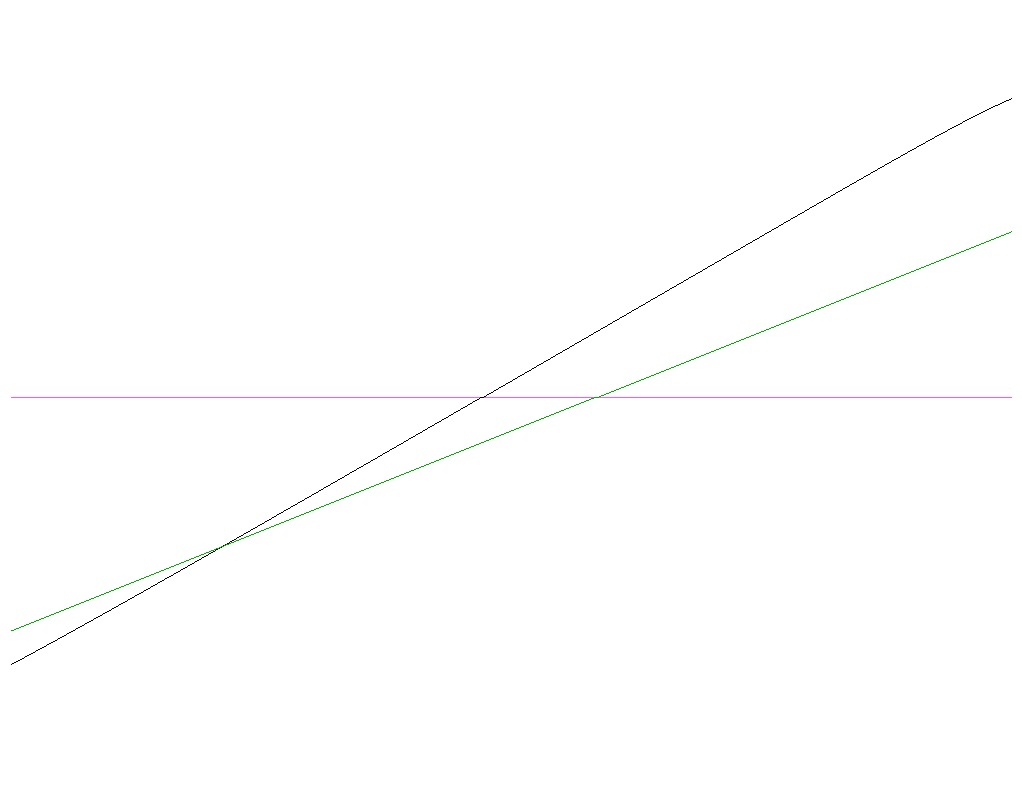}
\caption{The side $a=40$.}\label{f4b}
\end{center}
\end{subfigure}
\begin{subfigure}{0.45\textwidth}
\begin{center}
\includegraphics[width=0.75\textwidth]{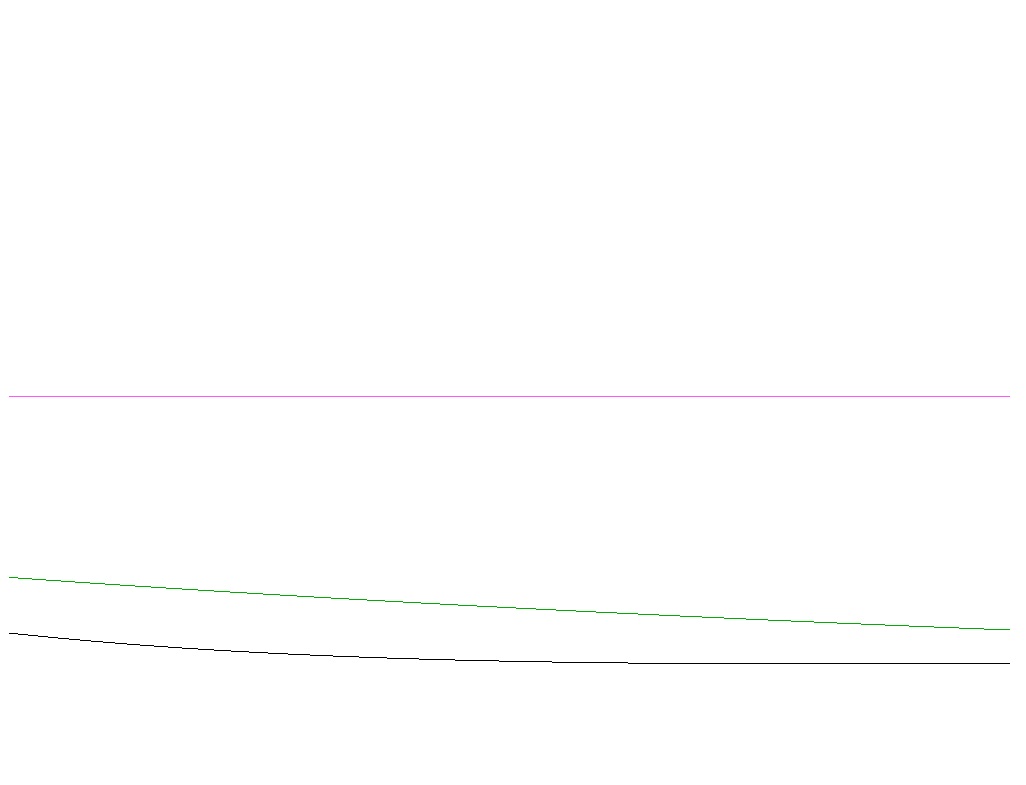}
\caption{The side $b=0.25$.}\label{f4c}
\end{center}
\end{subfigure}
\begin{subfigure}{0.45\textwidth}
\begin{center}
\includegraphics[width=0.75\textwidth]{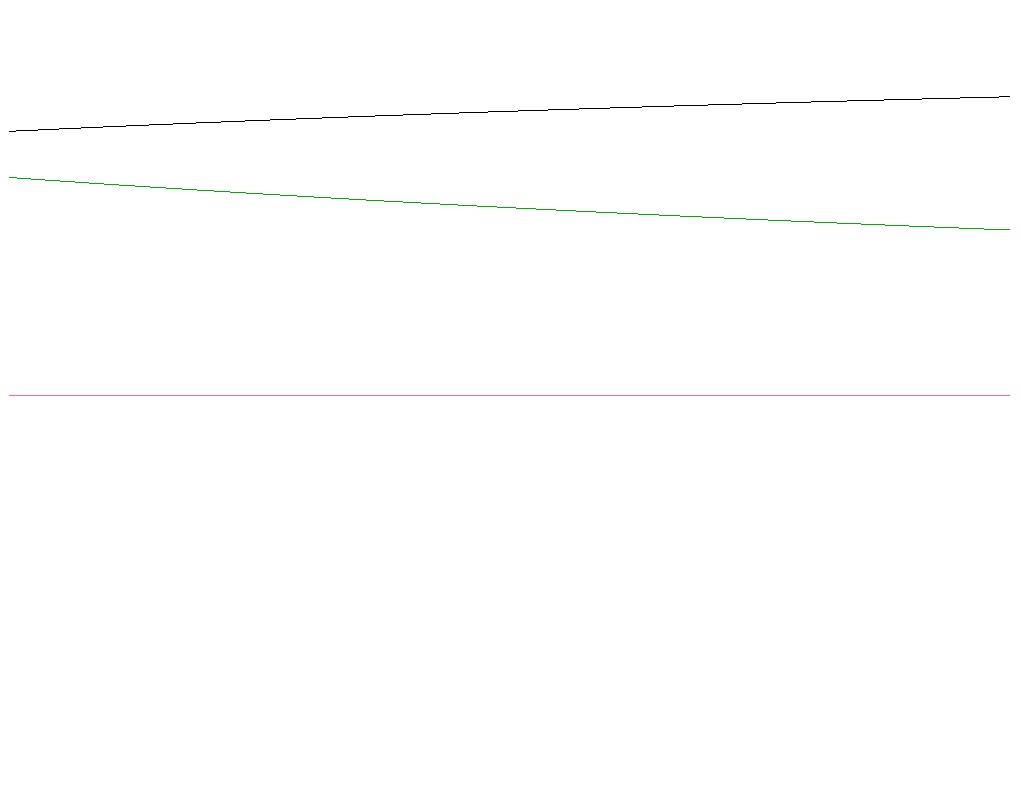}
\caption{The side $b=0.5$.}\label{f4d}
\end{center}
\end{subfigure}
\caption{The $(a,b)$-plane}
\end{figure}

When a=27.5 and b increases, the graphs go up and cyan gets to 0 level
first. When a=40 and b increases, the graphs go up and black gets to 0
level first. When b=0.25 and a varies, both graphs are below 0. When
b=0.5 and a varies, both graphs are above 0.
Moreover, all functions whose graphs we see in the figure are
monotone.

Denote our rectangle by $R$ and its boundary by $B$. Let $G:R\to\R^2$
be given by $G=(F^3(y_l)-y_f, F^2(y_r)-y_l)$. Then we see that the
image under $G$ of the side $b=0.25$ is in the third quadrant, of the
side $b=0.5$ in the first quadrant, of the side $a=27.5$ in the first,
second and third quadrants, and of the side $a=40$ in the third,
fourth and first quadrants. Therefore, $G|_B$ is a homeomorphism onto
its image and $G(B)$ is a Jordan curve with the origin in the bounded
component of its complement. Thus, it remains to prove the following
lemma.

\begin{lemma}\label{toplemma}
Let $E$ be a set homeomorphic to the unit closed disk $D$, with the
boundary $T$. Let $G:E\to\R^2$ be a continuous map, which maps $T$
homeomorphically onto its image, so $G(T)$ is a Jordan curve. Let $C$
be the bounded component of $\R^2\setminus G(T)$. Then $G(E)$ contains
$C$.
\end{lemma}

\begin{proof}
By the Jordan-Schoenflies Theorem and an elementary construction we
can precompose and postcompose $G$ via homeomorphisms, so that we get
a continuous map $H:D\to\R^2$, which is the identity on the unit
circle $S$. Suppose some point $x\in D\setminus S$ does not belong to
$H(D)$. Then the same it true for all points from some open
neighborhood of $x$, so by postcomposing $H$ with the map from the
punctured unit disk onto its boundary via rays emanating from $x$ we
get a retraction of $D$ onto $S$. But such retraction does
not exist, a contradiction. This completes the proof.
\end{proof}

\subsubsection{One measure served by one critical point}
\label{sec:one}
Now we exemplify a case 1.b.ii in \eqref{map} maps. Let us try to find parameters for which it occurs. In
Figure~\ref{f6} we present bifurcation diagram for $b=0.38$ where $a$ varies from $26.5$ to 27.\footnote{Bifurcation diagram describes limiting behavior of trajectories, showing attractors of the system. In Figure~\ref{f6} we do it by taking as initial points critical points of the map. Then we calculate 20000 iterations (to get close to the limiting behavior), and draw next 100 iterations. As we deal with the map with negative Schwarzian derivative to show all attractors it is sufficient to use only critical points of the map.}  There, for a large interval of values of $a$ there is a
periodic cycle of length 3 of intervals. One critical point is in this
cycle, while the other is outside it. Moreover, this is the only
invariant cycle of intervals. As $a$ varies, it should go through the usual
unimodal bifurcations \cite{de2012one,collet1983positive}, so there are values of $a$ for which
there is an absolutely continuous invariant measure supported on this
cycle. For those values of $a$ we have case 1.b.ii.

\begin{figure}[h!]
\begin{center}
\includegraphics[width=110truemm]{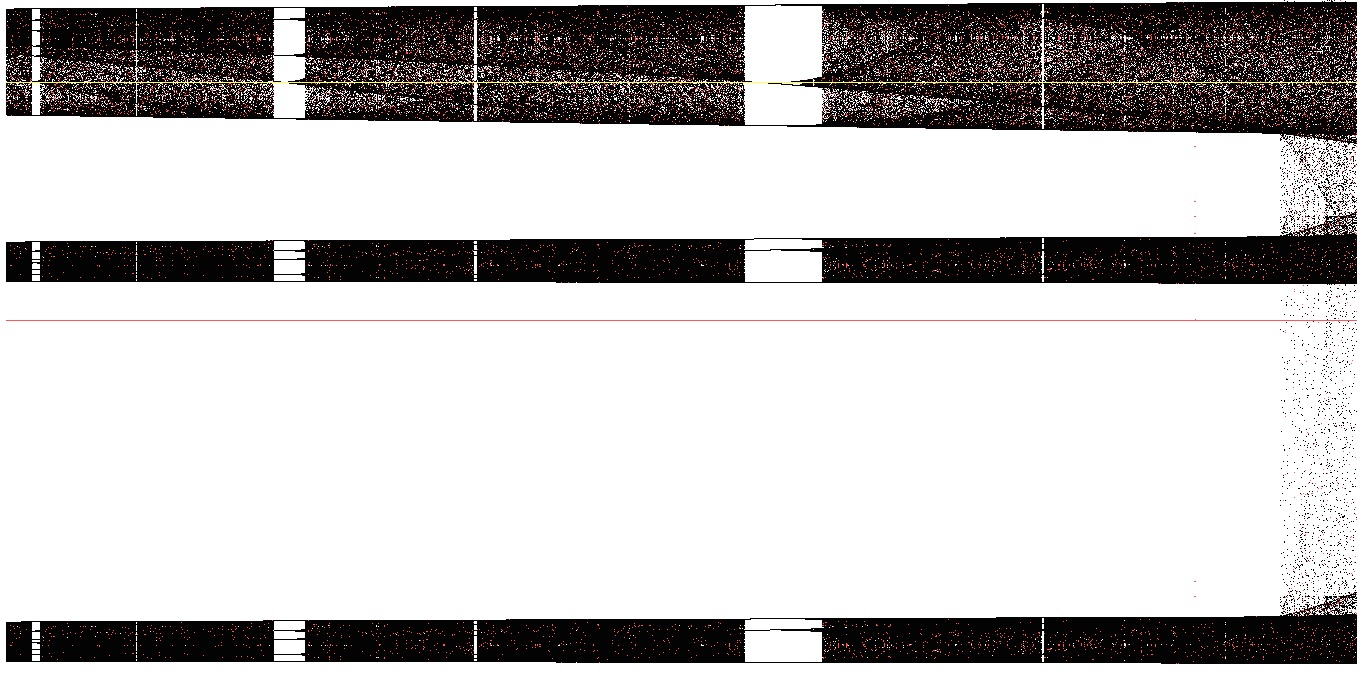}
\caption{Here $b=0.38$ and $a$ (on the horizontal axis) varies from
  $26.5$ to $27$. The positions of critical points are marked in
  red and yellow.}\label{f6}
\end{center}
\end{figure}

\subsection{Two cycles of intervals}
\label{sec:2cycles}
Now we focus on the case when there are two
disjoint cycles of intervals (case 2). Because our map $F$ has
negative Schwarzian derivative, in each of those cycles we can find a
critical point of $F$. Therefore, on each interval of the cycle the map
$F^n$, where $n$ is the length of the cycle, is unimodal. As
the parameters vary, we typically see the usual unimodal
bifurcations.

\subsubsection{Attracting periodic orbit coexisting with other invariant cycle}
\label{sec:2apo}
Fix $a=26$ and take $b$ from $0.386$ to
$0.39$. In our example, see Figure~\ref{f1z}, one of the invariant cycles of intervals has an
attracting periodic orbit and a critical point in the immediate basin
of attraction of this orbit. In the other cycle we will find the other
critical point (case 2.a or 2.b). As we mentioned, as the parameter
$b$ varies, we see there the usual unimodal bifurcations. In
particular, for many values of $b$ there is another attracting periodic orbit (case 2.a) and for many other values of $b$ there is an invariant absolutely continuous measure (case 2.b).
In Figure~\ref{f1z} we see that for $b$ from $0.386$ to
almost $0.388$, there are two disjoint cycles of intervals, one of
period 4 and the other of period 6, each of which contains a critical
point of $F$. In the period 6 cycle there is an attracting periodic
orbit of period 6 (see also Figures~\ref{f1b} and \ref{f1b00}).
For $b = 0.3868$, in the period 4 cycle there is an attracting
periodic orbit of period 4 (Figures~\ref{f1a} and~\ref{f1a00}).

\begin{figure}[h!]
\begin{center}
\includegraphics[width=120truemm]{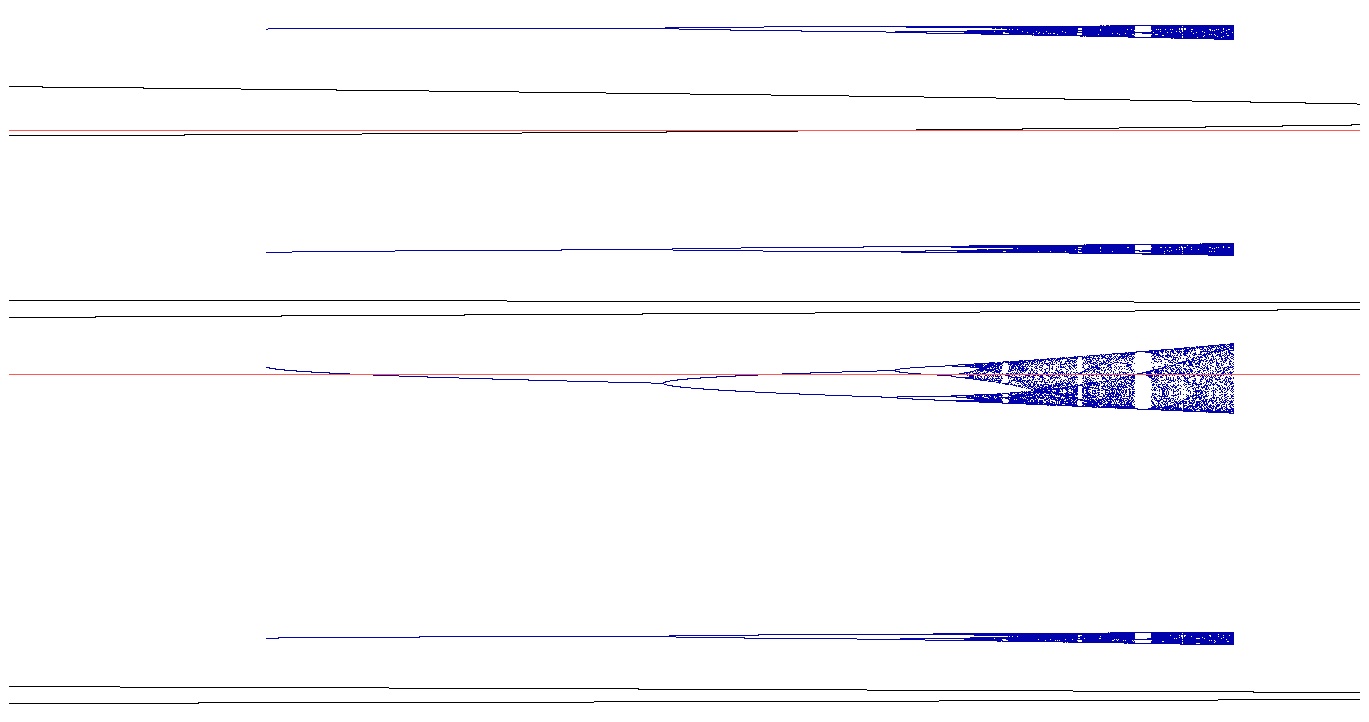}
\caption{Bifurcation diagram for $a=26$ and $b$ (on the horizontal
  axis) varies from $0.386$ to $0.39$. The attracting sets in the
  cycles of intervals are in black (period 6 cycle) and blue (period 4
  cycle). The positions of critical points are marked in
  red.}\label{f1z}
\end{center}
\end{figure}

\begin{figure}[h!]
\centering
\begin{subfigure}[t]{0.4\textwidth}
\begin{center}
\includegraphics[width=0.9\textwidth]{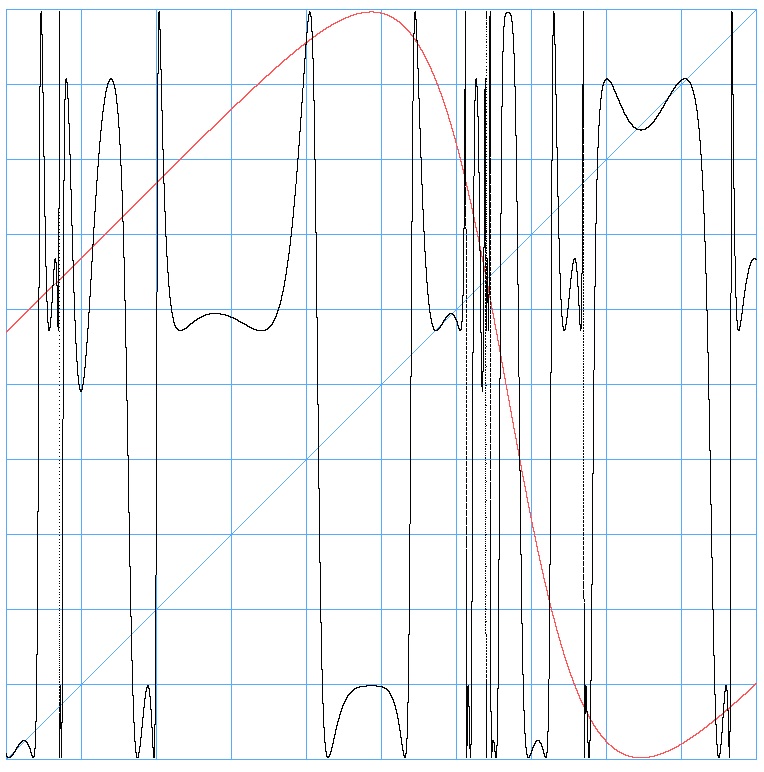}
\caption{Graphs of $F$ (red) and $F^6$ (black), 
 $x\in (-0.453,0.2236)$.}\label{f1b}
\end{center}
\end{subfigure}\qquad
\begin{subfigure}[t]{0.4\textwidth}
\begin{center}
\includegraphics[width=0.9\textwidth]{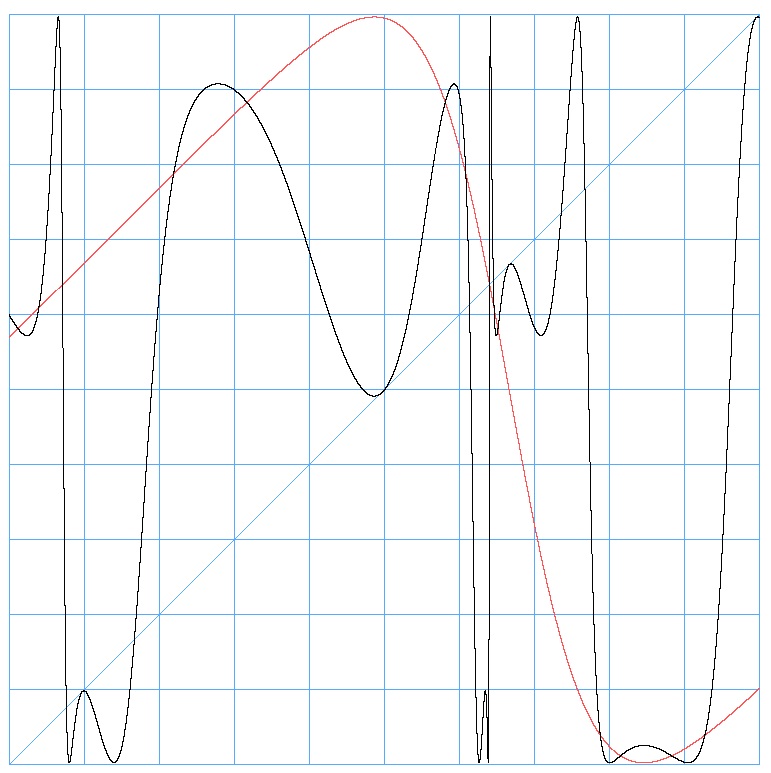}
\caption{Graphs of $F$ (red) and $F^4$ (black), 
  $x$ is between $-0.453$ and $0.2236$.}
  \label{f1a}
\end{center}
\end{subfigure}
\begin{subfigure}[t]{0.4\textwidth}
\begin{center}
\includegraphics[width=0.9\textwidth]{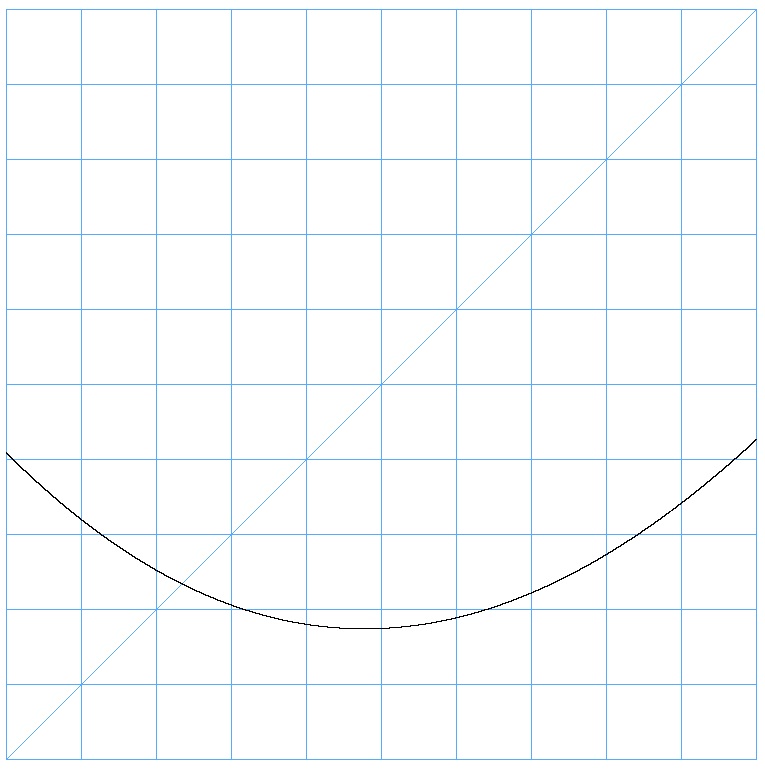}
\caption{The graph of $F^6$, 
$x\in (0.115,0.13)$. The critical point of $F^6$ visible here is the
  right critical point of $F$. It is in the immediate basin of
  attraction of a fixed point of $F^6$.}\label{f1b00}
\end{center}
\end{subfigure}\qquad
\begin{subfigure}[t]{0.4\textwidth}
\begin{center}
\includegraphics[width=0.9\textwidth]{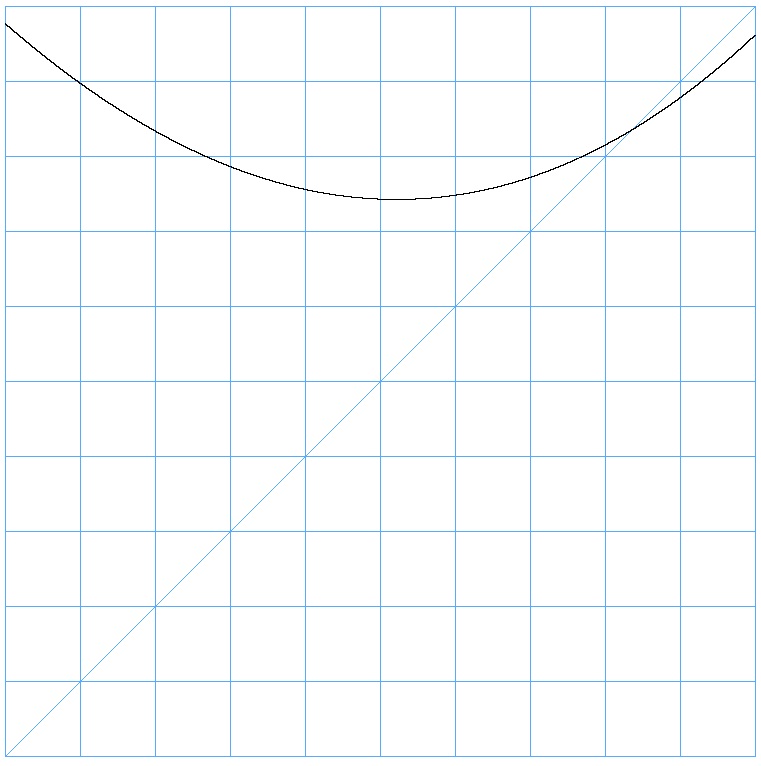}
\caption{The graph of $F^4$, 
$x\in (-0.13,-0.115)$. The critical point of $F^4$ visible here is
  the left critical point of $F$. It is in the immediate basin of
  attraction of a fixed point of $F^4$.}\label{f1a00}
\end{center}
\end{subfigure}
\caption{Two disjoint cycles of intervals for $a=26$, $b=0.3868$.}
\end{figure}

\subsubsection{Two measures}
\label{sec:2measures}
Let us find case 2.c. Figures~\ref{f5a} and~\ref{f5b} show
bifurcation diagrams for $a=27$ and $a=27.01$ respectively, as $b$
varies from $0.386$ to $0.388$. In those figures the trajectories of
both critical points (after skipping first 20000 points, the next 100
points are plotted) are shown in black and red. Of course, there are a
lot (countably many) values of $b$ for which the relevant critical
point is preperiodic (but not periodic). We see that as we move from
$a=27$ to $a=27.01$, the red part of the bifurcation diagram moves to
the right, while the black part moves to the left. This means that for
the suitably chosen preperiodic combinatorial patterns of the
trajectories, there is a value of $a$ between $27$ and $27.01$ for
which both critical points follow those patterns. For this value of
$a$ there are two absolutely continuous invariant measures with
disjoint supports (in the red and black areas).

\begin{figure}[h!]
\begin{center}
\includegraphics[width=110truemm]{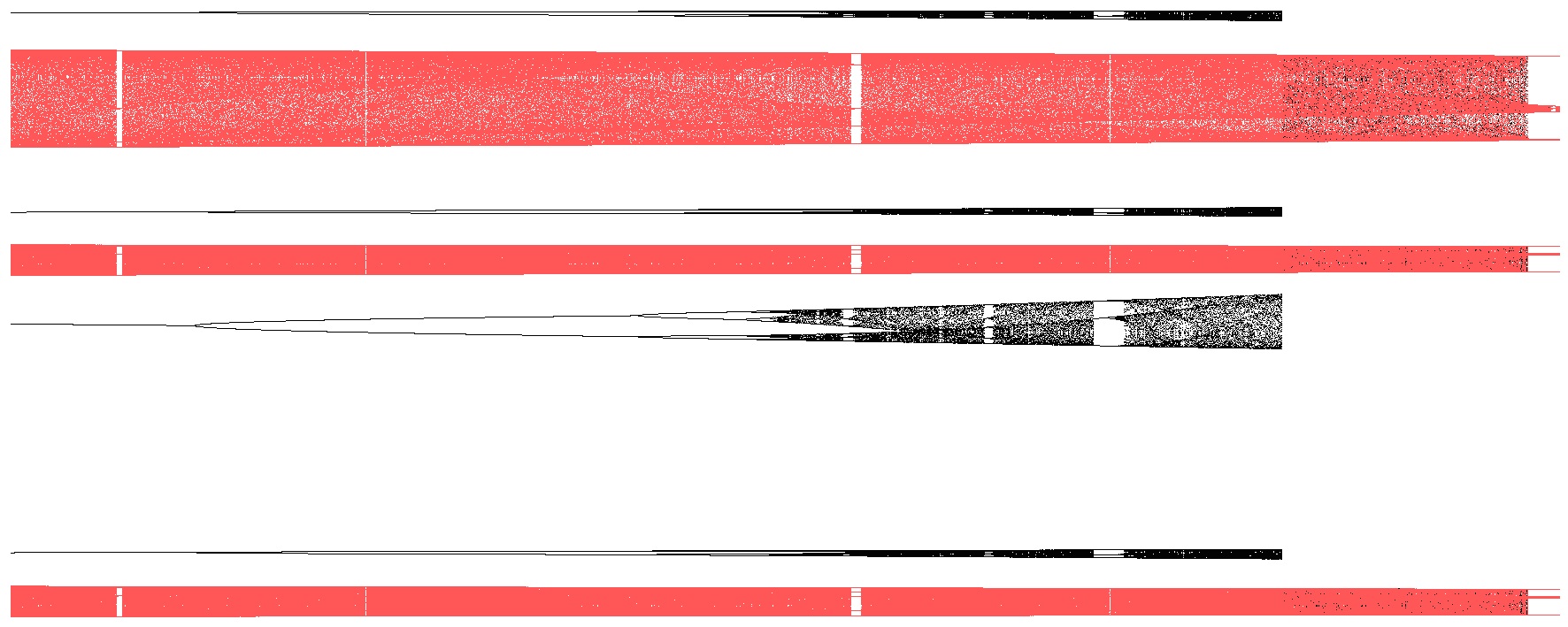}
\caption{Here $a=27$ and $b$ (on the horizontal axis) varies from
  $0.386$ to $0.388$.}\label{f5a}
\end{center}
\end{figure}

\begin{figure}[h!]
\begin{center}
\includegraphics[width=110truemm]{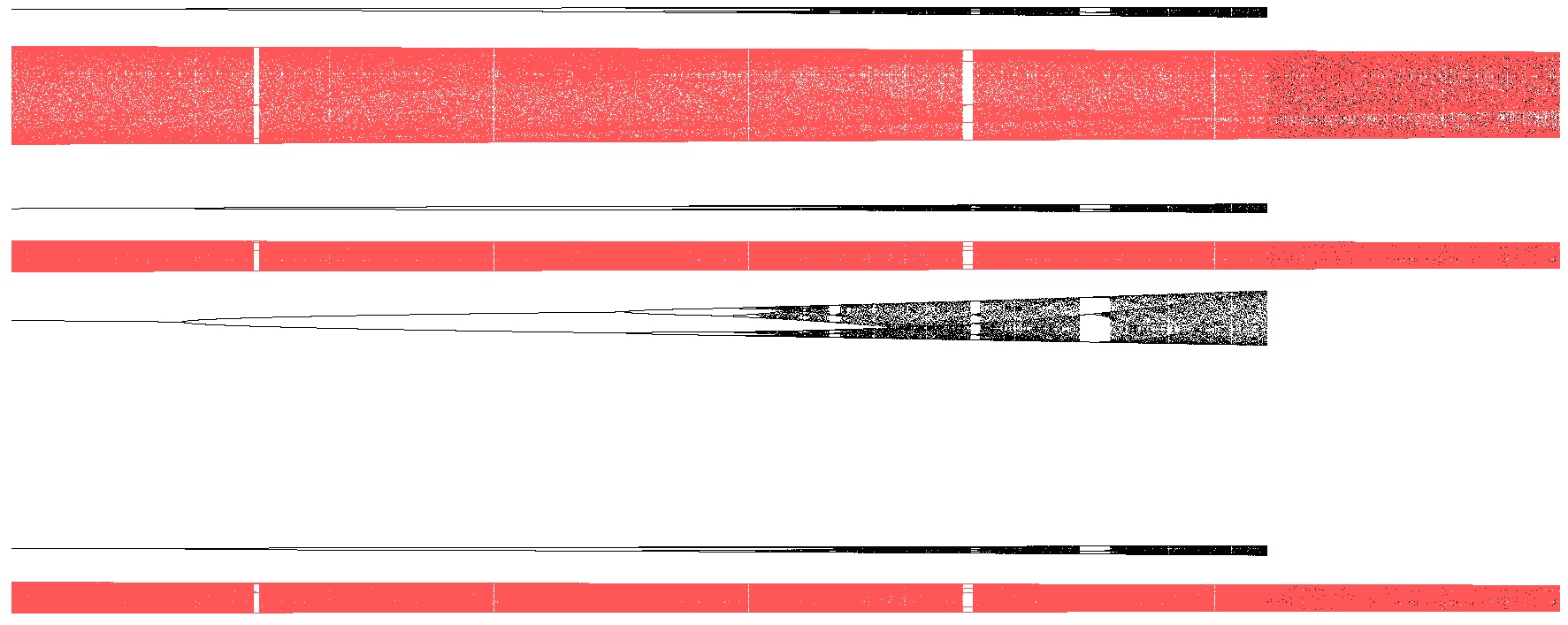}
\caption{Here $a=27.01$ and $b$ (on the horizontal axis) varies from
  $0.386$ to $0.388$.}\label{f5b}
\end{center}
\end{figure}

\subsection{Example observables in economics}
\label{sec:cost}

Here we present two economic/machine learning notions for which there exist observables that can be applied to our theoretical results, these are the social cost and the regret. 

\subsubsection{Social cost}

In the congestion game setting, a particularly meaningful observable is the \emph{total cost function} or, equivalently the \emph{social cost}, which quantifies the overall cost incurred by the population. In our two-route congestion game, if a fraction $x$ of the agents choose route 1, then the social cost is
$
\Phi(x) \;=\; N^2\bigl(\alpha\,x^2 \;+\; \beta\,(1-x)^2\bigr).
$
This function $\Phi(x)$ is a strictly convex potential with a global minimum at the Nash equilibrium $x=b.$

Despite the fact that for sufficiently large $a$ 
the fraction $x_n$ may not converge to the equilibrium, the \emph{natural invariant measure} discussed earlier ensures that the \emph{time-averaged} cost still converges to a well-defined value. Intuitively, while the system may enter limit cycles or chaotic orbits for large $a$ (and hence never settle exactly at $x=b$), this total cost function will exhibit convergent long-term averages.

Figure~\ref{fig: potential_dynamics} provides a visual demonstration of this phenomenon. 
In the \textbf{left column}, we plot the cost function $\Phi\bigl(x_m\bigr)$ as we iterate the MWU map for increasing values of $a$; a line segment connects $\Phi(x_m)$ to $\Phi(x_{m+1}) \;=\;\Phi\bigl(f_{MW}(x_m)\bigr),$ and the line color evolves from blue (early iterations) to red (later iterations). 
The \textbf{middle column} shows the corresponding cobweb diagrams for the map $f_{MW}$, and the \textbf{right column} depicts the time evolution of the fraction of agents on route 1, $x_m.$ From top to bottom, the parameter $a$ increases while $b=0.4$ remains fixed. When $a$ is small (top row), the trajectory converges to the Nash equilibrium. As $a$ increases, we encounter successively more complex behavior: period-2 orbits, period-4 orbits, and eventually chaos (bottom row). Notably, as shown by the horizontal dashed green line in the right column, the orbits time average remains exactly $b$, despite the lack of pointwise convergence to the equilibrium. Here, we initialized with $x_0 = 0.9.$

\begin{figure}[h!]
\begin{center}
\includegraphics[width=0.8\textwidth]{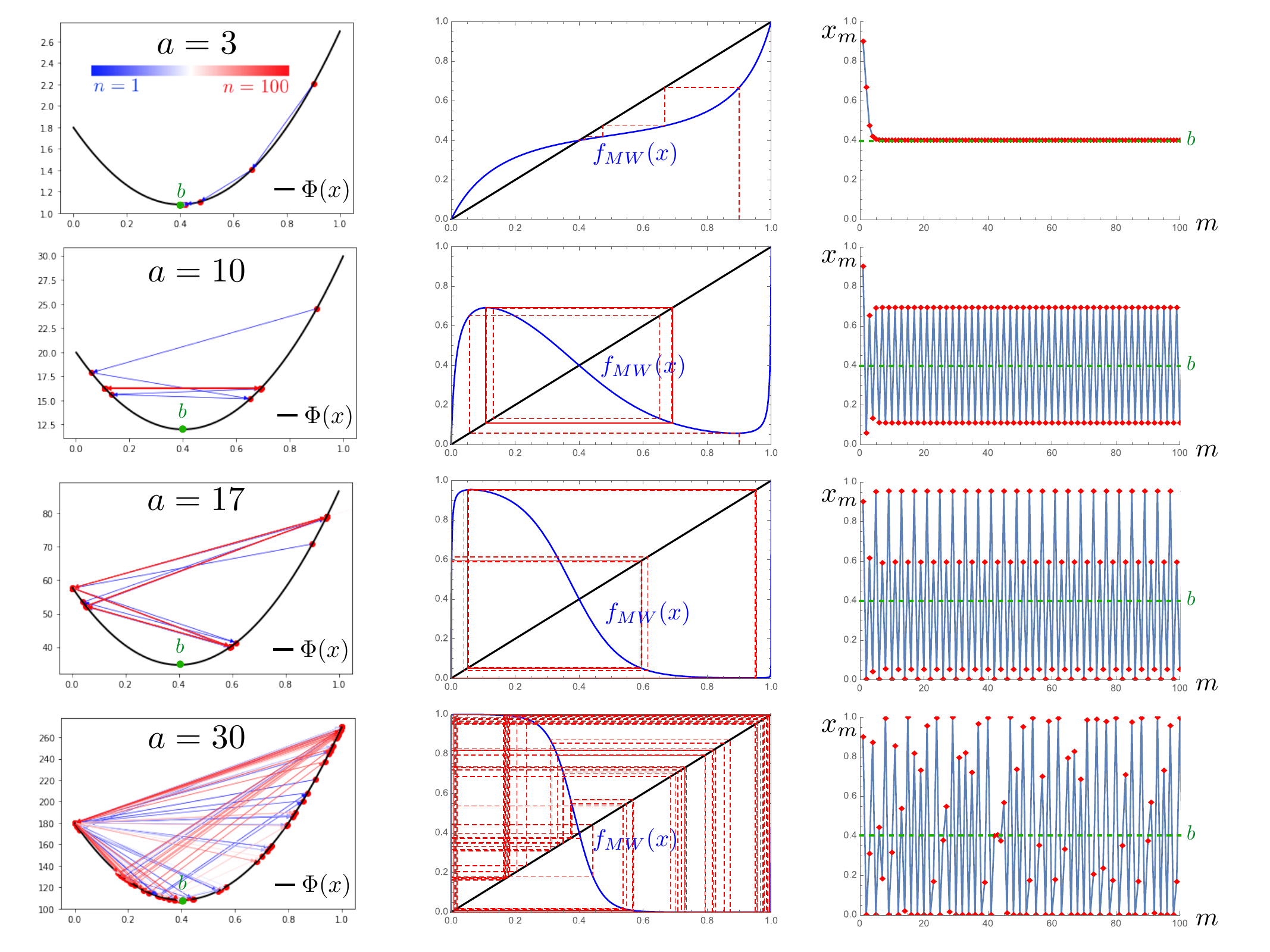}
\caption{An illustration of the cost function, cobweb diagrams, and trajectories for the fraction of agents using route 1, under MWU. 
From top to bottom, the parameter $a$ increases while $b=0.4$ is fixed. 
The \textbf{left column} depicts $\Phi(x_m)$ joined to $\Phi(x_{m+1})=\Phi\bigl(f_{MW}(x_m)\bigr)$ (colored from blue to red as time progresses). 
The \textbf{middle column} shows cobweb diagrams of the map $f_{MW},$ and the \textbf{right column} presents the corresponding time evolution of $x_m$. 
For small $a$, the trajectory converges to the Nash equilibrium at $x=b.$ 
As $a$ increases, we see a transition to period-2 orbits, period-4 orbits, and finally chaotic behavior. 
Crucially, the horizontal green dashed lines indicate the time average of $x_m,$ which equals the equilibrium $b,$ even when $x_m$ does not converge to $b$. 
Here, $x_0 = 0.9.$ 
Since the cost function is an observable, its time average will converge to a well-defined value guaranteed by the main result of this work.}\label{fig: potential_dynamics}
\end{center}
\end{figure}

\subsubsection{Regret}

The regret is a notion that allows us to evaluate the performance of a randomizing algorithm. It measures the performance of an algorithm compared to the best fixed action in hindsight.
Fix cost vectors $c_1, c_2, \ldots, c_T$. The (expected) regret of an algorithm choosing actions $x_1, x_2, \ldots, x_T$ is
\[
 R_T := \sum_{n=1}^T \mathbb{E}_{a_n \sim x_n} c_n(a_n) - \min_{a \in A} \sum_{n=1}^T c_n(a),
\]
where $\mathbb{E}_{a_n \sim x_n} c_n(a_n)$ expresses the expected cost of the algorithm in time period $n$, when an action $a_n \in A$ is chosen according to the probability distribution $x_n$.

By the results of \cite{CFMP2019} we know that in our setting, the limit of the time-average regret is equal to the size of the population of agents $N$ times the limit of the observable $(x-b)^2$ (provided this limit exists), that is
$ \lim_{T \to \infty} R_T/T = N \left( \lim_{T \to \infty} \frac{1}{T} \sum_{n=1}^T (x_n - b)^2 \right)$.
Because the observable $(x-b)^2$ is continuous we get from Theorem \ref{thm1} that the limit of the time average regret is equal to the space average of the observable $(x-b)^2$ with respect to an invariant probability measure. And since the observable $(x-b)^2$ is convex we obtain from Theorem \ref{thm2} that the limit of the time average regret is the lowest in the Nash equilibrium play.


\section{Conclusion \& future work}
\label{sec:discussion}

Our work applies natural invariant measures as a powerful tool for analyzing the long-term behavior of chaotic game dynamics, specifically focusing on the Multiplicative Weights Update (MWU) algorithm. We have demonstrated that, even in a seemingly simple two-strategy congestion game, MWU can exhibit the full spectrum of behaviors possible in one-dimensional dynamical systems.  This includes scenarios with unique or multiple absolutely continuous invariant measures, as well as coexisting chaotic and stable (periodic) behaviors.  Our results provide, for the first time, a comprehensive \emph{statistical} understanding of MWU's chaotic regimes, going beyond the mere identification of Li-Yorke chaos \cite{liyorke}. We characterize the frequency of strategy profiles, sensitivity to initial conditions, and long-term averages of payoffs, all through the lens of these invariant measures.
The intersection of game theory, dynamical systems, computer science, and statistics is rich with open questions, particularly when considering learning dynamics in games. We propose the following directions for future work:


    \textbf{Higher-Dimensional Games and Algorithms:}  A crucial next step is to extend the analysis to higher-dimensional games.  This presents significant challenges:
 \begin{itemize}
        \item
        \textit{Characterizing Invariant Measures:}  In higher dimensions, the structure of invariant measures can be vastly more complex (e.g., fractal measures, SRB measures \cite{young2002srb}).  Developing computational techniques to approximate and analyze these measures is a key challenge. Are there efficient algorithms (perhaps leveraging techniques from computational topology \cite{edelsbrunner2010computational} or statistical physics) to estimate key properties of these measures, such as their dimension or entropy?
        \item \textit{Beyond MWU:} While MWU is a fundamental algorithm, many other learning algorithms are used in practice (e.g., Fictitious Play \cite{brown1951iterative}, Follow the Regularized Leader \cite{shalev2012online}, Replicator Dynamics \cite{taylor1978evolutionary}, Q-learning \cite{watkins1992q}, and various forms of reinforcement learning).  How do the long-term statistical properties of these algorithms differ?  Can we develop a taxonomy of learning algorithms based on the types of invariant measures they generate?
    \end{itemize}

    \textbf{Computational Complexity and Approximation:}
      \begin{itemize}
        \item The computation of invariant measures and the analysis, even in the relatively simple settings, can present a computational challenge. The development of practical tools that can be applied to solve the open questions is crucial.
        \item What are the computational complexities of characterizing the long-term behavior and relevant statistical properties of these measures? How do these depend on the description/size of the game?
    \end{itemize}

These research directions highlight the exciting possibilities at the intersection of these fields. By combining the tools of dynamical systems, game theory, computer science, and statistics, we can gain a deeper understanding of complex, adaptive systems and develop more effective algorithms and mechanisms for a wide range of applications.

\begin{acks}
Fryderyk Falniowski acknowledge the support of the National Science Centre, Poland, grant 2023/51/B/HS4/01343.
Thiparat Chotibut was supported by Thailand Science Research and
  Innovation Fund Chulalongkorn University (IND\_FF\_69\_258\_2300\_062). 
Misiurewicz was partially supported by grant number 426602 from the Simons Foundation.
\end{acks}


\bibliographystyle{ACM-Reference-Format}
\bibliography{ms}

\end{document}